\documentclass[12pt]{amsart}
\usepackage[latin1]{inputenc}
\usepackage[T1]{fontenc}
\usepackage[francais]{babel}

\usepackage{amssymb,amsmath,amsthm,amscd,verbatim}
\usepackage[all]{xy}

\newcommand{\codim}{\operatorname{codim}}
\newcommand{\by}[1]{\overset{#1}{\longrightarrow}}
\newcommand{\iso}{\by{\sim}}
\renewcommand{\phi}{\varphi}
\renewcommand{\epsilon}{\varepsilon}
\newcommand{\surj}{\to\!\!\!\!\!\to}

\newcommand{\inj}{\hookrightarrow}
\renewcommand{\lim}{\varprojlim}
 
\newcommand{\osi}{\overset{\sim}{\longleftarrow}}

\newcommand{\Grp}{\operatorname{\mathbf{Grp}}}

\newcommand{\Ker}{\operatorname{Ker}}
\newcommand{\Coker}{\operatorname{Coker}}

\newcommand{\Sm}{\operatorname{\mathbf{Sm}}}

\newcommand{\Frac}{\operatorname{Frac}}
\newcommand{\Hom}{\operatorname{Hom}}

\newcommand{\Cores}{\operatorname{Cor}}
\newcommand{\Spec}{\operatorname{Spec}}

\newcommand{\Ab}{\operatorname{\mathbf{Ab}}}
\newcommand{\Nis}{{\operatorname{Nis}}}

\newcommand{\IM}{{\operatorname{Im}}}
\newcommand{\DM}{{\operatorname{\mathbf{DM}}}}
\newcommand{\CM}{{\operatorname{CM}}}
\newcommand{\HI}{{\operatorname{HI}}}

\newcommand{\et}{{\operatorname{\acute{e}t}}}
\newcommand{\eff}{{\operatorname{eff}}}

\renewcommand{\o}{{\operatorname{o}}}

\newcommand{\NR}{{\operatorname{NR}}}
\newcommand{\proj}{{\operatorname{proj}}}
\newcommand{\nr}{{\operatorname{nr}}}
\newcommand{\st}{{\operatorname{st}}}
\newcommand{\op}{{\operatorname{op}}}
\newcommand{\Br}{{\operatorname{Br}}}
\newcommand{\nab}{{\operatorname{nab}}}
\renewcommand{\neg}{{\operatorname{neg}}}
\newcommand{\Ind}{\operatorname{Ind}}
\newcommand{\Inv}{\operatorname{Inv}}

\newcommand{\fl}{{\operatorname{fl}}}

\newcommand{\sA}{\mathcal{A}}
\newcommand{\sB}{\mathcal{B}}
\newcommand{\sC}{\mathcal{C}}

\newcommand{\sG}{\mathcal{G}}
\newcommand{\sO}{\mathcal{O}}

\newcommand{\sH}{\mathcal{H}}

\newcommand{\sP}{\mathcal{P}}
\newcommand{\sF}{\mathcal{F}}

\newcommand{\A}{\mathbb{A}}

\newcommand{\Q}{\mathbb{Q}}
\newcommand{\Z}{\mathbb{Z}}
\newcommand{\G}{\mathbb{G}}

\newcommand{\bP}{\mathbb{P}}

\newcommand{\cf}{{\it cf.}\ }
\newcommand{\ie}{{\it i.e.}\ }

\newenvironment{thlist}{\begin{list}{\rm{(\roman{enumi})}}%
{\usecounter{enumi}}}%
{\end{list}}

\newtheorem{Th}{Th\'eor\`eme}
\newtheorem{thm}{Th\'eor\`eme}[section]
\newtheorem{prop}[thm]{Proposition}
\newtheorem{Cor}{Corollaire}
\newtheorem{cor}[thm]{Corollaire}

\newtheorem{lemme}[thm]{Lemme}
\newtheorem{construct}[thm]{Construction}

\theoremstyle{definition}
\newtheorem{defn}[thm]{D\'efinition}
\theoremstyle{remark}
\newtheorem{rem}[thm]{Remarque}
\newtheorem{qn}[thm]{Question}
\newtheorem{ex}[thm]{Exemple}
\newtheorem{exs}[thm]{Exemples}

\numberwithin{equation}{section}

\begin{document}

\title[Modules de cycles et classes non ramifi\'ees]{Modules de cycles et classes non ramifi\'ees sur un espace classifiant}
\author{Bruno Kahn}
\address{IMJ-PRG\\Case 247\\4 place
Jussieu\\75252 Paris Cedex 05\\France}
\email{bruno.kahn@imj-prg.fr}
\author{Nguyen Thi Kim Ngan}
\address{Faculty of Natural Sciences\\Thu Dau Mot University\\ Binh Duong\\Vietnam}
\email{nguyen.t.k.ngan.vn@gmail.com}
\date{14 novembre 2014}
\begin{abstract} Let $G$ be a finite group of exponent $m$ and let $k$ be a field of characteristic prime to $m$, containing the $m$-th roots of unity. For any Rost cycle module $M$ over $k$, we construct exact sequences
\[\begin{CD}
0\to \Inv_k^\nr(G,M_n)\to \Inv_k(G,M_n)@>(\partial_{D,g})>> \displaystyle\bigoplus_{(D,g)} \Inv_k(D,M_{n-1})
\end{CD}\]
where $\Inv_k(G,M_n)$ is Serre's group of invariants of $G$ with values in $M_n$,  $\Inv_k^\nr(G,M_n)$ is its subgroup of unramified invariants, and the ``residue" morphisms $\partial_{D,g}$ are associated to pairs $(D,g)$ where $D$ runs through the subgroups of $G$ and $g$ runs through the homomorphisms $\mu_m\to G$ whose image centralises $D$. This allows us to recover results of Bogomolov and Peyre on the unramified cohomology of fields of invariants of $G$,  and to generalise them. 
\end{abstract}
\subjclass[2010]{20G15, 14C15, 20J06}
\maketitle

\tableofcontents

\section*{Introduction} Soit $G$ un groupe alg\'ebrique lin\'eaire sur un corps $k$, et soit $M=M_*$ un module de cycles sur $k$ au sens de Rost \cite{rost}. Suivant Serre \cite[d\'ef. 1.1]{GMS}, on d\'efinit le \emph{groupe des invariants de $G$ de degr\'e $n$ \`a valeurs dans $M$} 
\[\Inv_k(G,M_n)\]
pour tout entier $n\in \Z$: il s'agit du groupe des transformations na\-tu\-rel\-les $\phi_K:H^1(K,G)\to M_n(K)$ o\`u $K$ d\'ecrit la cat\'egorie des extensions de $k$. On peut aussi parler d'invariants non ramifi\'es au sens de \cite[33.9]{GMS}: ils forment un sous-groupe $\Inv_k^\nr(G,M_n)$, dont la trivialit\'e est une condition n\'ecessaire pour une solution positive au probl\`eme de Noether.

En g\'en\'eral le calcul de $\Inv_k^\nr(G,M_n)$ est difficile, et celui de $\Inv_k(G,M_n)$ est un peu plus facile. Le but de cet article est de r\'eduire le calcul du premier \`a celui du second, dans certains cas particuliers. Le r\'esultat principal est le suivant: 

\begin{Th}\label{t1}  Supposons $G$ fini constant, d'ordre premier \`a la caract\'eristique de $k$. Supposons de plus que $k$ contienne les racines $m$-i\`emes de l'unit\'e, o\`u $m$ est l'exposant de $G$. Alors, \`a tout sous-groupe $D$ de $G$ et \`a tout homomorphisme $g:\mu_m \to Z_G(D)$ (centralisateur de $D$ dans $G$), on peut associer un ``homomorphisme r\'esidu"
\[\partial_{D,g}:\Inv_k(G,M_n)\to \Inv_k(D,M_{n-1})\]
tel que la suite
\[\begin{CD}
0\to \Inv_k^\nr(G,M_n)\to \Inv_k(G,M_n)@>(\partial_{D,g})>> \displaystyle\bigoplus_{(D,g)} \Inv_k(D,M_{n-1})
\end{CD}\]
soit exacte.
\end{Th}

Les hypoth\`eses sont essentielles, en particulier celle sur les racines de l'unit\'e. Pour $G$ et $k$ quelconques, on a encore un complexe comme  ci-dessus, mais les $\partial_{D,g}$ ne sont sans doute pas suffisants  pour d\'etecter les invariants non ramifiés. En trouver la bonne généralisation est un probl\`eme ouvert.\footnote{Mathieu Florence a fait une suggestion dans ce sens au premier auteur, quand $G$ est fini constant mais que $k$ ne contient pas assez de racines de l'unit\'e.}

\begin{Cor}\label{c1} Gardons les hypoth\`eses du th\'eor\`eme \ref{t1} et d\'efinissons
\[\Inv_k^\nab(G,M_n) = \Ker\left(\Inv_k(G,M_n)\to \bigoplus_A \Inv_k(A,M_n)\right)\]
o\`u $A$ d\'ecrit les sous-groupes ab\'eliens de $G$. Alors la suite exacte du th\'eor\`eme \ref{t1} se raffine en une suite exacte
\[\begin{CD}
0\to \widetilde{\Inv}^\nr_k(G,M_n)\to \Inv_k^\nab(G,M_n)@>(\partial_{D,g})>>\displaystyle \bigoplus_{(D,g)} \Inv_k^\nab(D,M_{n-1})
\end{CD}\]
o\`u $\widetilde{\Inv}_k^\nr(G,M_n)$ d\'esigne le sous-groupe des invariants non ramifi\'es normalis\'es (triviaux sur le $G$-torseur neutre).
\end{Cor}

Ce corollaire peut \^etre consid\'er\'e comme une g\'en\'eralisation d'un th\'eor\`eme de Bogomolov \cite[th. 7.1]{ct-s}, qu'il permet de retrouver. Nous retrouvons aussi un th\'eor\`eme de Peyre  \cite[th. 1]{peyre1}.

Voici la strat\'egie de la d\'emonstration. D'apr\`es Totaro \cite[app. C]{GMS}, on a un isomorphisme  canonique
\[\Inv_k(G,M_n)= A^0(BG,M_n)\]
o\`u le groupe de droite est celui de \cite[\S 5]{rost}. Le sous-groupe $\Inv_k^\nr(G,M_n)$ du membre de gauche correspond \`a un sous-groupe $A^0_\nr(BG,M_n)$ du membre de droite. Tout ceci demande un explication, puisque $BG$ n'est pas bien d\'efini en tant que vari\'et\'e lisse: c'était l'objet de \cite{BG1} (notamment \S 6.4) dont la théorie est rappelée en \ref{s.rappels}. 

Les r\'esidus $\partial_{D,g}$ sont d\'efinis au \S \ref{res.geom} en termes de $BG$;  le th\'eor\`eme \ref{t1} et le corollaire \ref{c1} sont d\'emontrés au \S \ref{th.princ};  au \S \ref{bog.peyre}, on fait le raccord avec les r\'esultats de Bogomolov et de Peyre.

Voici une intuition qui peut \'eclairer le th\'eor\`eme \ref{t1}. Soit $X$ une $k$-vari\'et\'e lisse, et supposons donn\'ee une compactification lisse $X\subset \bar X$, telle que $D=\bar X-X$ soit un diviseur \`a croisements normaux. On a alors une suite exacte \cite[rem. 6.9]{kahn0}
\[0\to A^0(\bar X,M_n)\to A^0(X,M_n)\to \bigoplus_{i=1}^r A^0(D_i^\o,M_{n-1})\]
o\`u $D_i$ d\'ecrit l'ensemble des composantes irr\'eductibles de $D$ et $D_i^\o = D_i\setminus\bigcup_{j\ne i} D_j$. Imaginons que $BG$ soit repr\'esent\'e par une vari\'et\'e $X$ comme ci-dessus. Dans le th\'eor\`eme \ref{t1}, tout se passe comme si on pouvait alors trouver une compactification $\bar X$ tel que les $D_i^\o$ soient de la forme $BD$, o\`u $D$ d\'ecrit les sous-groupes de $G$\dots

Ce travail est une version remani\'ee et augment\'ee de la seconde partie de la th\`ese du second auteur (N.T.K.N.)  \cite{nganthese}, la première partie étant parue dans \cite{BG1}; il a été annoncé dans \cite{nganCR}. B.K. tient \`a remercier N.T.K.N. de l'avoir invit\'e \`a r\'ediger ce texte en commun.

\subsection*{Guide de lecture} On utilise le langage développé dans \cite{BG1} et rappelé au \S \ref{s.rappels}: si $G$ est un groupe algébrique linéaire et $F$ est un foncteur sur les variétés lisses vérifiant des axiomes simples, on peut donner un sens à $F(BG)$.

La section \ref{res.geom} introduit l'outil principal de l'article: les r\'esidus g\'eom\'etriques $\partial_{D,g}$. La section \ref{s9} encha\^\i ne en montrant que les classes sur $BG$, non ramifi\'ees au sens de \cite{ct-o}, sont non ramifi\'ees au sens des r\'esidus g\'eom\'etriques.

La section \ref{th.princ} est le c\oe ur de ce travail. On y d\'emontre que r\'eciproquement, si $G$ est fini constant d'exposant $m$ et que $k$ est de caract\'eristique premi\`ere \`a $m$ et contient les racines $m$-i\`emes de l'unit\'e, les classes sur $BG$, non ramifi\'ees au sens des r\'esidus g\'eom\'etriques, le sont au sens de \cite{ct-o}. La raison pour laquelle les r\'esidus g\'eom\'etriques suffisent dans ce cas semble \^etre que, si $G$ est le groupe de Galois d'une extension $L/K$ o\`u $K\supset k$ et que $w$ est une valuation discr\`ete de rang $1$ sur $L$, le groupe d'inertie de $w$ est \emph{central} dans son groupe de d\'ecomposition.

Enfin, la section \ref{s11} reformule le r\'esultat pr\'ec\'edent de mani\`ere universelle, tandis que la section \ref{bog.peyre} l'utilise pour retrouver (et g\'en\'eraliser) des th\'eor\`emes de Bogomolov et Peyre \cite{bogomolov,peyre1}.

\section{R\'esidus g\'eom\'etriques}\label{res.geom}

Nous introduisons ici des ``r\'esidus g\'eom\'etriques'', inspir\'es par les travaux de Peyre et de Voevodsky. On fixe un corps de base $k$.

\subsection{Rappels de \cite{BG1}} \label{s.rappels} Notons  $\Sm_\fl$ la catégorie des $k$-variétés lisses, les morphismes étant les morphismes plats. Pour $r\ge 1$, on note comme dans \cite[déf. 2.3]{BG1} $S_r$ l'ensemble  des morphismes de $\Sm_\fl$, union de l'ensemble des projections d'espace total un fibré vectoriel et de l'ensemble des immersions ouvertes $j:U\to X$ de coniveau $\ge r$ (par définition, cela signifie que $\codim_X(X-U)\ge r$). Notons d'autre part $\Grp$ la catégorie des $k$-groupes algébriques linéaires (morphismes: les homomorphismes de groupes algébriques); dans \cite[déf. 2.14]{BG1}, on a défini une suite compatible de foncteurs
\begin{equation}
B_r:\Grp\to S_r^{-1} \Sm_\fl, \quad r\ge 1
\end{equation}
où le membre de droite est la localisation à la Gabriel-Zisman de $\Sm_\fl$ relativement à $S_r$. 

Pour référence ultérieure, rappelons la construction de $B_r$. On choisit une représentation linéaire $V$ de $G$ très fidèle de coniveau $\ge r$: cela signifie \cite[déf. 2.7]{BG1} que $V$ contient un ouvert $U$ $G$-invariant, espace total d'un $G$ torseur et tel que $\codim_V(V-U)\ge r$. L'existence de $V$ est assurée par un théorème de Totaro \cite[rem. 1.4]{totaro}. On définit $B_rG$ (essentiellement) comme $U/G\in S_r^{-1}\Sm_\fl$: le ``lemme sans nom'' permet de montrer que $B_rG$ est indépendant du choix de $V$ à isomorphisme unique près, et fonctoriel en $G$.

Un foncteur $F$ de $\Sm_\fl$ vers une catégorie $\sC$ est dit \emph{homotopique et pur en coniveau $\ge r$} s'il inverse les morphismes de $S_r$ (en particulier, $F$ est invariant par homotopie). De manière équivalente, $F$ induit un foncteur $S_r^{-1}\Sm_\fl\by{F_r} \sC$, et on pose
\[F(BG) = F_r(B_rG).\]

On renvoie \`a \cite[\S 3.2 et \S 5]{BG1} pour des exemples de tels foncteurs: cohomologie étale, cohomologie motivique, cohomologie de cycles de Rost.

Rappelons également la définition suivante \cite[déf. 3.8]{BG1}:

\begin{defn}\label{tildeF(BG)}
Soit $F$ un foncteur de $\Sm_\fl$ \`a valeurs dans la cat\'egorie des groupes abéliens. Pour tout $G\in\Grp$, on pose
\[\tilde{F}(BG)=\Ker (F(B\epsilon))\]
o\`u $\epsilon$ est l'idempotent $G\to \Spec(k) \to G$. C'est la \emph{partie réduite} de $F(BG)$.
\end{defn}


Dans \cite{BG1} on n'a utilis\'e que la fonctorialit\'e plate sur les $k$-sch\'emas lisses: cela permet par exemple de d\'efinir \'economiquement le foncteur $G\mapsto CH^n(BG)$ en n'utilisant que la contravariance ``facile" des groupes de Chow, comme Totaro dans \cite{totaro}. Ici nous aurons aussi besoin de fonctorialit\'e pour certaines immersions ferm\'ees.

\subsection{La construction $F_{-1}$ de Voevodsky} 

\begin{defn}\label{defn-Voe}
Soit $F$ un foncteur contravariant de $\Sm_\fl$ vers la cat\'egorie $\Ab$ des groupes ab\'eliens. On d\'efinit pour $X\in \Sm_\fl$ (\cf \cite[Lect. 23]{MVW}):
\begin{equation}\label{res-Voe}
F_{-1}(X)=\Coker (F(X\times \A^1) \to F(X\times (\A^1-\{0\})))
\end{equation}
et on note 
\[\partial: F(X\times \G_m)\to F_{-1}(X).\]
Si $F$ est contravariant pour les immersions ferm\'ees r\'eguli\`eres, on note
\begin{equation}\label{eq:8.2}
s:F(X\times \G_m)\to F(X)
\end{equation}
le morphisme d\'efini par la section unit\'e de $\G_m$.
\end{defn}

Le lemme suivant se d\'emontre sans difficult\'e:

\begin{lemme}\label{F_{-1}(BG)}
Si $F$ est homotopique et pur en coniveau $\geq r$, $F_{-1}$ l'est aussi.\qed
\end{lemme}

\begin{lemme}\label{iso-res}
Si $F$ est invariant par homotopie et contravariant pour les immersions ferm\'ees r\'eguli\`eres, alors 
\[(s,\partial):F(X\times \G_m)\to F(X)\oplus F_{-1}(X)\]
est un isomorphisme pour tout $X$.
\end{lemme}

\begin{proof}
Comme $F$ est invariant par homotopie, $F(X)\iso F(X\times \A^1)$. De plus, on a le diagramme commutatif suivant
\begin{displaymath}
\xymatrix{
F(X\times \A^1)\ar[r]& F(X\times \G_m)\ar[dl]_{s}\ar[r]^{\partial}& F_{-1}(X)\ar[r]& 0\\
F(X)\ar[u]^{\wr}\ar @/_/[ur]_{p^*}
}
\end{displaymath}
o\`u $p:X\times \G_m \to X$ est la projection sur $X$. Comme $s\circ p^*=id$, on a
\[F(X\times \G_m)\iso F(X)\oplus F_{-1}(X).\]
\end{proof}

\begin{exs}[$F_{-1}(X)$ pour certains foncteurs $F$]\label{exF_{-1}(X)}\
\begin{enumerate}
	\item Pour $F(X)=H^i(X,\Z(n))$, on a
\[F_{-1}(X)=H^{i-1}(X,\Z(n-1))\]
\cite[23.1]{MVW}. 
On a le m\^eme r\'esultat pour la cohomologie \'etale et la cohomologie motivique \'etale \ie
\begin{itemize}
	\item  Si $F(X)=H^i_{\et}(X,M)$, alors $F_{-1}(X)=H^{i-1}_{\et}(X,M(-1))$;
	\item  Si $F(X)=H^i_{\et}(X,\Z(n))$, alors $F_{-1}(X)=H^{i-1}_{\et}(X,\Z(n-1))$.
\end{itemize}

 \item Soit $F(X)=A^0(X,M_n)$, où $(M_n)_{n\in\Z}$ est un module de cycles au sens de Rost \cite[\S 5]{rost}. On a une suite exacte de localisation
 \begin{multline*}
 A^0(X\times \A^1,M_n)\to A^0(X\times \G_m,M_n) \by{(4)} A^0(X,M_{n-1})\\
 \by{(5)} A^1(X\times \A^1,M_n)\by{(6)} A^1(X\times \G_m,M_n)
 \end{multline*}
 
Par le lemme \ref{iso-res} et l'invariance par homotopie de la cohomologie de cycles, (6) est injectif scind\'e, donc (5) est nul et (4) est surjectif, soit
 \[F_{-1}(X)=A^0(X,M_{n-1}).\]
 
Ce calcul a implicitement utilis\'e la contravariance de la cohomologie de cycles pour les immersions ferm\'ees \cite[\S 12]{rost}.
\end{enumerate}
\end{exs}

\subsection{R\'esidus g\'eom\'etriques ``universels"}\label{res.univ}

Soit $G$ un $k$-groupe alg\'ebrique lin\'eaire et soit $H$ un sous-groupe ferm\'e de $G$; soit $r\ge 1$.  Dans \cite[Ex. 2.18]{BG1} on a construit un morphisme dans $S_r^{-1}\Sm_\fl$
\begin{equation}\label{eq:G/H2}
G/H\to B_r H.
\end{equation}

\begin{construct}\label{c.univ} Soit $m$ un entier inversible dans $k$. Soit $X$ un sch\'ema lisse sur $k$. Soit $F$ un foncteur homotopique et pur en coniveau $\geq r$. Le \emph{morphisme r\'esidu universel}
\begin{equation}\label{residu}
\partial_m:F(X\times B\mu_m) \to F_{-1}(X)
\end{equation}
est d\'efini de la mani\`ere suivante:

En utilisant \eqref{eq:G/H2} avec $G=\G_m$ et $H=\mu_m$ et en remarquant que $\G_m/\mu_m\iso \G_m$, on obtient une fl\`eche canonique $\G_m\to B_r \mu_m$ dans $S_r^{-1}\Sm_\fl$, d'o\`u une
composition: 
\[\partial_m:F(X\times B\mu_m) \to F(X\times \G_m) \by{\partial} F_{-1}(X).\]
\end{construct}

De mani\`ere \'equivalente, soit $U$ un $\G_m$-torseur lin\'eaire de coniveau $\geq r$ \cite[déf. 2.7]{BG1}. Consid\'erons le diagramme:
\begin{equation}\label{eq:8.1}
\xymatrix{
& \G_m\ar[d] & U\times \G_m \ar[l]_{\pi_1} \ar[r]^{\pi_2}\ar[d]& U\ar[d]\\
\G_m & \G_m/\mu_m \ar[l]^{\sim}_{\times m} & (U\times \G_m)/\mu_m \ar[l]_{\bar{\pi}_1}\ar[r]^{\bar{\pi}_2}&U/\mu_m.
}
\end{equation}

La ligne du bas induit un diagramme
\begin{multline*}
F(X\times (U/\mu_m)) \to F(X\times (U\times \G_m)/\mu_m)\\
\osi F(X\times (\G_m/\mu_m)) \osi F(X\times \G_m) \to F_{-1}(X).
\end{multline*}
D'o\`u on d\'eduit \eqref{residu}.

Nous aurons besoin du lemme suivant, qui r\'esulte de la description du morphisme bord d'une longue suite exacte de cohomologie:

\begin{lemme}\label{l:partial} Soit $M$ un module de cycles.  Notons  $\partial_0$ le r\'esidu relatif \`a la valuation discr\`ete sur $k(t)=k(\A^1)$ correspondant \`a l'origine de $\A^1$. Notons d'autre part $\delta$ le bord de la suite exacte de localisation relative \`a l'immersion ouverte $\G_m\allowbreak\inj \A^1$. Alors le diagramme
\[\xymatrix{
A^0(\G_m,M_n)\ar[r]^\delta\ar@{^{(}->}[d] & A^0(k,M_{n-1})\ar[d]^{=}\\
M_n(k(t))\ar[r]^{\partial_0} & M_{n-1}(k)
}\]
est commutatif.\qed
\end{lemme}

\begin{lemme}\label{residuP}
Soit $M$ un module de cycles sur $k$, et soit $n\in\Z$. Alors le morphisme r\'esidu \eqref{residu} pour $A^0(-,M_n)$ est com\-pa\-tible avec celui consid\'er\'e par Peyre dans \cite[(13) p. 207]{peyre1}. \\
Plus pr\'ecis\'ement, soient $\chi:\mu_m\inj k^*$ le caract\`ere canonique de $\mu_m$, $W_\chi$ la repr\'esentation fid\`ele de dimension $1$ correspondante, $B_\chi$  la valuation discr\`ete de rang $1$ sur $k(X\times W_\chi)=k(X)(T)$ d\'efinie par le diviseur $T=0$ et  $A_\chi$ sa trace sur $k(X\times W_\chi)^{\mu_m}$. Alors on a un diagramme commutatif
\[
\xymatrix{
A^0(X \times B\mu_m,M_n) \ar[r]^{\partial_m}\ar@{^(->}[d]& A^0(X,M_{n-1})\ar@{^(->}[d]\\
M_n(k(X\times W_\chi)^{\mu_m}) \ar[r]^{\partial_{A_{\chi}}}& M_{n-1}(k(X)).
}
\]
\end{lemme}

\begin{proof} Par fonctorialit\'e, on peut remplacer $X$ par son corps des fonctions, donc (quitte \`a changer de corps de base) supposer que $X=\Spec k$.

Soit $U=\A^r-\{0\}$; faisons op\'erer $\G_m$ sur $\A^r$ et donc sur $U$ par homoth\'eties. Choisissons un point rationnel $x\in U(k)$: ce point d\'efinit un morphisme $\G_m$-\'equivariant $\phi: \G_m \to U$, d'o\`u un morphisme:
\[\bar{\phi}:\G_m \osi \G_m/\mu_m \to U/\mu_m.\]

Soit $\gamma:\G_m \to U\times \G_m$ le transpos\'e du graphe de $\phi$: avec les notations de \eqref{eq:8.1}, c'est une section $\mu_m$-\'equivariante de $\pi_1$ telle que $\pi_2\circ \gamma=\phi$. En prenant les quotients par $\mu_m$, $\gamma$ induit une section $\bar{\gamma}$ de $\bar{\pi}_1$ telle que $\bar{\pi}_2\circ \bar{\gamma}=\bar{\phi}$, ce qui implique que la composition:
\[A^0((U/\mu_m),M_n) \by{\bar{\phi}^*} A^0(\G_m,M_n) \by{\delta} M_{n-1}(k)\]
est \'egale \`a $\partial_m$, o\`u $\delta$ est le morphisme bord pour la suite exacte de localisation relative \`a l'immersion ouverrte $\G_m\inj \A^1$.

Finalement, on a $\phi(\G_m)=L-\{0\}$ o\`u $L$ est la droite $kx$. L'assertion r\'esulte maintenant du lemme \ref{l:partial}.
\end{proof}

\subsection{R\'esidus g\'eom\'etriques ``\`a la Peyre"}

\begin{defn}\label{residuBG}
 Soit $F$ comme dans la construction \ref{c.univ}. Soient $G\in\Grp$, $D \subset G$ un sous-groupe ferm\'e et $g: \mu_m \to Z_G(D)$ un homomorphisme, o\`u $Z_G(D)$ d\'esigne le centralisateur de $D$ dans $G$. On note $\phi:D\times \mu_m \to G$ l'homomorphisme d\'efini par $\phi(d,i)=d.g(i)$. On d\'efinit un morphisme:
\[\partial_{D,g}: F(BG) \by{\phi^*} F(B(D\times \mu_m))\iso F(BD\times B\mu_m) \by{\partial_m} F_{-1}(BD)\]
o\`u $\partial_m$ est comme dans \eqref{residu}, l'isomorphisme provient de l'isomorphisme $B_r(D\times \mu_m)\simeq B_r D\times B_r\mu_m$ de \cite[lemme 2.24]{BG1} et $F_{-1}(BD)$ est d\'efini par le lemme \ref{F_{-1}(BG)}.
\end{defn}

\begin{prop}\label{diag-comm}
$\partial_{D,g}$ est canonique et fonctoriel en $F$.\qed
\end{prop}

\begin{exs}
D'apr\`es les exemples \ref{exF_{-1}(X)}, on a des r\'esidus  sui\-vants:
\begin{align*}
H^n_{\et}(BG,\mu_m^{\otimes j}) &\to H^{n-1}_{\et}(BD,\mu_m^{\otimes (j-1)}),\\
H^n_{\et}(BG,\Z(q)) &\to H^{n-1}_{\et}(BD,\Z(q-1)),\\
H^n(BG,\Z(q)) &\to H^{n-1}(BD,\Z(q-1)),\\
A^0(BG,M_n) &\to A^0(BD,M_{n-1}).
\end{align*}
\end{exs}

\begin{rem} Supposons $G$ fini constant, d'exposant $e$ premier \`a la caract\'eristique de $k$. Alors tout homomorphisme $g$ comme ci-dessus a pour image un sous-groupe cyclique d'ordre $m'$ divisant $e$, donc se factorise par $\mu_{m'}$ via la surjection $\mu_m\to \mu_{m'}$; de plus, $\mu_{m'}$ est constant. Si $g':\mu_{m'}\to Z_G(D)$ est l'homomorphisme induit, on voit tout de suite que $\partial_{D,g}=\partial_{D,g'}$ en comparant deux suites exactes de Kummer. On peut donc se limiter aux valeurs de $m$ divisant $e$ et telles que $\mu_m\subset k$.
\end{rem}

\subsection{Le r\'esidu d'un cup-produit} Le r\'esultat principal de ce num\'ero (th\'eor\`eme \ref{t.cup}) ne sera pas utilis\'e dans la suite; il est n\'eanmoins tr\`es utile pour des calculs concrets.

\subsubsection{D\'ecomposition de la diagonale} 
Soient $X\in \Sm(k)$ et $f:X\to \G_m$ ($f\in \Gamma(X,\sO_X^*)$). Soit $\Gamma_f:X\to X\times \G_m$ le graphe de $f$. Soit $F:\Sm(k)^\op\to \Ab$ un foncteur homotopique et pur en coniveau $\geq c$. On a le diagramme suivant:
\begin{displaymath}
\xymatrix{
F(X\times \A^1) \ar[r] & F(X\times \G_m) \ar[r]\ar[d]^{\Gamma^*_f} & F_{-1}(X) \ar[r]& 0\\
& F(X)\ar[ul]^{\sim}
}
\end{displaymath}
Si $f=1$, alors $\Gamma^*_f=s^F$ (la section de \eqref{eq:8.2}) et donc on a
\begin{equation}\label{eq-iso-res}
F(X\times \G_m)\iso F(X)\oplus F_{-1}(X)
\end{equation}
d'apr\`es le lemme \ref{iso-res}. Alors $\Gamma^*_f -\Gamma^*_1=0$ sur $F(X\times \A^1)$, donc induit un morphisme
\begin{equation}\label{eq-F-1aF}
\{f\}^*:F_{-1}(X)\to F(X).
\end{equation}

Consid\'erons maintenant la diagonale $\G_m \by{\Delta} \G_m\times \G_m$. On a le morphisme
\[(1_X \times \Delta)^*: F(X\times \G_m\times \G_m) \to F(X\times \G_m).\]
Par transport de structure, il d\'efinit un morphisme $\tilde{\Delta}$:
\begin{displaymath}
\xymatrix{
F(X\times \G_m\times \G_m) \ar[r]^{(1_X\times \Delta)^*}\ar[d]^{\wr} & F(X\times \G_m)\ar[d]^{\wr}\\
F(X)\oplus 2F_{-1}(X)\oplus F_{-2}(X)\ar[r]^{\qquad \tilde{\Delta}} & F(X)\oplus F_{-1}(X)
}
\end{displaymath}
o\`u $F_{-2}=(F_{-1})_{-1}$ et l'isomorphisme de gauche est obtenu en appliquant deux fois le lemme \ref{iso-res}. 

On va calculer $\tilde{\Delta}$ dans le cas particulier o\`u $F$ provient d'un foncteur sur $\DM:=\DM_\Nis^{\eff,-}(k)$ \cite[Lect. 14]{MVW}. Le lemme suivant justifie la notation \eqref{eq-F-1aF}:

\begin{lemme}\label{surDMeffgm}
Si $F$ provient d'un foncteur sur $\DM$ et $X=\Spec k$, alors $\{f\}^*$ est induit par le cup-produit par $\{f\}\in K^M_1(k)=H^1(k,\Z(1))$ pour $f\in k^*$.
\end{lemme}

\begin{proof}
Dans ce cas, $\Gamma_f^*-\Gamma_1^*$ provient de 
\[M(X)\by{\Gamma_f-\Gamma_1} M(X\times \G_m)=M(X)\oplus M(X)(1)[1] \by {pr} M(X)(1)[1]\]

 Si $X=\Spec k$, d'o\`u $f\in k^*$, alors
\[\Gamma_f-\Gamma_1 \in \Hom(\Z,\Z(1)[1])\cong K^M_1(k)\ \text{\cite[4.2]{MVW}}.\]

On v\'erifie que cet isomorphisme identifie $\Gamma_f-\Gamma_1$ \`a $\{f\}\in K^M_1(k)\cong k^*$ (voir \cite[preuve de 4.4]{MVW}).
\end{proof}

\begin{prop}\label{formule-res}
Supposons que $F$ se factorise en 
\[\Sm(k)^\op\by{M^\op} \DM^\op \by{F} \Ab.\]
Alors $\tilde{\Delta}$ est \'egale \`a 
\begin{displaymath}
\left(
\begin{array}{ccc}
1 & 0& 0\\
0 & \sum & \{-1\}
\end{array}
\right).  
\end{displaymath}
\end{prop}

\begin{proof}
On choisit la d\'ecomposition $M(\G_m)=\Z\oplus \Z(1)[1]$ donn\'ee par le point $1\in \G_m$. Cela donne exactement \eqref{eq-iso-res} sur $F(X\times \G_m)$ et $F_{-1}(M(X))=F(M(X)(1)[1])$. On a le diagramme commutatif suivant:
\begin{displaymath}
\xymatrix{
F(M(X)\otimes M(\G_m\times \G_m)) \ar[r]^{\quad (1_X \times \Delta)^*}\ar[d]^{\wr} & F(M(X)\otimes M(\G_m))\ar[d]^{\wr}\\
F(M(X)\otimes (\Z\oplus \Z(1)[1])^{\otimes 2}) \ar[r] \ar[d]^{\wr}& F(M(X)\otimes (\Z\oplus \Z(1)[1]))\ar[d]^{\wr}\\
F(M(X))\oplus 2F_{-1}(M(X))\oplus F_{-2}(M(X))\ar[r]^{\qquad \tilde{\Delta}} & F(M(X))\oplus F_{-1}(M(X))
}
\end{displaymath}

Dans ce cas, $\tilde{\Delta}$ est induit par
\[M(\Delta): \Z\oplus \Z(1)[1] \to \Z \oplus 2\Z(1)[1]\oplus \Z(2)[2]\]
et on a
\[\Hom(\Z(j)[j],\Z(i)[i])=
\begin{cases}
\Hom(\Z,\Z(i-j)[i-j])=K^M_{i-j}(k)& \text{si } i\geq j,\\
                        0&  \text{si } i<j.
\end{cases}\]

D'apr\`es \cite[lem. 7.4 et cor. 7.9 (b)]{HK}, on trouve que $\tilde{\Delta}$ est de la forme
\begin{displaymath}
\left(
\begin{array}{ccc}
1 & 0& 0\\
0 & \sum & \{-1\}
\end{array}
\right)  
\end{displaymath}
o\`u 
$\sum:2F_{-1}(M(X))\to F_{-1}(M(X))$ est induit par $\Z(1)[1]\to 2\Z(1)[1]$ et $\{-1\}:F_{-2}(M(X))\to F_{-1}(M(X))$ est induit par $\{-1\}:\Z(1)[1]\to\Z(2)[2]$ (\cf lemme \ref{surDMeffgm}), ce qui correspond \`a la formule de \cite[cor. 7.9(b)]{HK}. 
\end{proof}

\subsubsection{R\'esidus g\'eom\'etriques et cup-produits}

Soient $F,G,H$ dans la cat\'egorie des foncteurs contravariants de $\Sm(k)$ vers les groupes ab\'eliens et supposons que pour tous sch\'emas $X,Y\in \Sm(k)$, on ait un produit externe:
\begin{equation}\label{prod-ex}
F(X)\otimes G(Y) \to H(X\times Y)
\end{equation}
bifonctoriel en $(X,Y)$. D'o\`u un produit interne:
\[F(X)\otimes G(X) \to H(X\times X)\to H(X)\]
o\`u la derni\`ere fl\`eche est donn\'ee par la diagonale $X\to X\times X$.

Consid\'erons le diagramme commutatif de suites exactes:
\begin{displaymath}
\xymatrix{
F(X\times \A^1)\otimes G(Y) \ar[r]\ar[d]& F(X\times \G_m)\otimes G(Y) \ar[r]\ar[d]& F_{-1}(X)\otimes G(Y)\ar[r]\ar @{-->}[d]_{\exists}&0\\
H(X\times Y \times \A^1)\ar[r]& H(X\times Y \times \G_m) \ar[r]& H_{-1}(X\times Y)\ar[r]& 0
}
\end{displaymath}

On en d\'eduit un morphisme
\begin{equation}\label{prod-ex1}
F_{-1}(X) \otimes G(Y) \to H_{-1}(X\times Y),
\end{equation}
et de mani\`ere analogue
\begin{equation}\label{prod-ex2}
F(X) \otimes G_{-1}(Y) \to H_{-1}(X\times Y).
\end{equation}
Pour $Y=X$, on a
\[F(X)\otimes G_{-1}(X)\oplus F_{-1}(X) \otimes G(X) \to H_{-1}(X\times X) \to H_{-1}(X).\]
Consid\'erons encore le diagramme commutatif de suites exactes suivant:
\begin{displaymath}
\xymatrix{
F(X\times \A^1)\otimes G_{-1}(Y) \ar[r]\ar[d]& F(X\times \G_m)\otimes G_{-1}(Y) \ar[r]\ar[d]& F_{-1}(X)\otimes G_{-1}(Y)\ar[r]\ar @{-->}[d]_{\exists}&0\\
H_{-1}(X\times Y \times \A^1)\ar[r]& H_{-1}(X\times Y \times \G_m) \ar[r]& H_{-2}(X\times Y)\ar[r]& 0
}
\end{displaymath}

On en d\'eduit un morphisme
\begin{equation}\label{prod-ex3}
F_{-1}(X) \otimes G_{-1}(Y) \to H_{-2}(X\times Y).
\end{equation}

\begin{thm}\label{t.cup}
Supposons que $F,G,H$ se factorisent par $\DM$. Soient $x\in F(X\times \G_m)$ et $y\in G(X\times \G_m)$; notons $xy$ leur cup-produit dans $H(X\times \G_m)$. Alors on a
\begin{equation}\label{eq-res-cup}
\partial^H(xy)=\partial^F(x)s^G(y)+s^F(x)\partial^G(y)+\{-1\}\partial^F(x)\partial^G(y)
\end{equation}
o\`u on note $\partial^F\dots$ pour garder la trace de $F,G,H$.
\end{thm}

\begin{proof}
Appliquant ce qui pr\'ec\`ede, on a des diagrammes commutatifs:
\begin{displaymath}
\xymatrix{
F_{-1}(X) \otimes G(Y\times \G_m) \ar[r]\ar[d]^{1\times \partial^G}& H_{-1}(X\times Y\times \G_m)\ar[d]^{\partial^{H_{-1}}}\\
F_{-1}(X) \otimes G_{-1}(Y) \ar[r]& H_{-2}(X\times Y)
}
\end{displaymath}
et
\begin{displaymath}
\xymatrix{
F_{-1}(X) \otimes G(Y\times \G_m) \ar[r]\ar[d]^{1\times s^G}& H_{-1}(X\times Y\times \G_m)\ar[d]^{s^{H_{-1}}}\\
F_{-1}(X) \otimes G(Y) \ar[r]& H_{-1}(X\times Y)
}
\end{displaymath}
o\`u $s^G,s^{H_{-1}}$ sont des sp\'ecialisations comme dans le lemme \ref{iso-res}.

Maintenant consid\'erons le diagramme commutatif:
\begin{displaymath}
\xymatrix{
F(X\times \G_m) \otimes G(Y\times \G_m) \ar[r]^{\sim \qquad}\ar[d] & (F(X)\oplus F_{-1}(X))\otimes (G(Y)\oplus G_{-1}(Y)) \ar[dd]\\
H(X\times \G_m \times Y\times \G_m) \ar[d]^{\wr} \\
H(X\times Y\times \G_m \times \G_m)\ar[d]^{(1_{X\times Y}\times \Delta)^*}\ar[r]^{\sim \qquad \qquad}& H(X\times Y)\oplus 2H_{-1}(X\times Y)\oplus H_{-2}(X\times Y)\ar[d]^{\tilde{\Delta}}\\
H(X\times Y\times \G_m)  \ar[r]^{\sim} & H(X\times Y)\oplus H_{-1}(X\times Y)
}
\end{displaymath}
o\`u la longue fl\`eche de droite est construite \`a partir de \eqref{prod-ex}, \eqref{prod-ex1}, \eqref{prod-ex2}, \eqref{prod-ex3}. D'o\`u le diagramme commutatif
\begin{displaymath}
\xymatrix{
F(X\times \G_m) \otimes G(Y\times \G_m)\ar[d]\ar[r]^{ \partial^F\times s^G+s^F\times \partial^G\quad}_{+\partial^F\times \partial^G\qquad} & \quad \qquad \text{$\begin{array}{cc}F_{-1}(X)\otimes G(Y)&\oplus  F(X)\otimes G_{-1}(Y)\\ \oplus  F_{-1}(X)\otimes G_{-1}(Y)\end{array}$}\ar[d]\\
H(X\times Y\times \G_m \times \G_m)\ar[d]^{(1_{X\times Y}\times \Delta)^*}\ar[r]& 2H_{-1}(X\times Y)\oplus H_{-2}(X\times Y)\ar[d]^{\tilde{\Delta}}\\
H(X\times Y\times \G_m)  \ar[r]^{\partial^H} & H_{-1}(X\times Y)
}
\end{displaymath}
Si $Y=X$, on a
\begin{displaymath}
\xymatrix{
F(X\times \G_m) \otimes G(X\times \G_m)\ar[d]\ar[r]^{ \partial^F\times s^G+s^F\times \partial^G\quad}_{+\partial^F\times \partial^G\qquad} & \quad \qquad \text{$\begin{array}{cc}F_{-1}(X)\otimes G(X)&\oplus  F(X)\otimes G_{-1}(Y)\\ \oplus  F_{-1}(X)\otimes G_{-1}(X)\end{array}$}\ar[d]\\
H(X\times X\times \G_m \times \G_m)\ar[d]^{(1_{X\times X}\times \Delta)^*}\ar[r]& 2H_{-1}(X\times X)\oplus H_{-2}(X\times X)\ar[d]^{\tilde{\Delta}}\\
H(X\times X\times \G_m)  \ar[r]^{\partial^H}\ar[d]^{(\Delta_X\times 1)^*} & H_{-1}(X\times X)\ar[d]^{(\Delta_X)^*}\\
H(X\times \G_m)  \ar[r]^{\partial^H} & H_{-1}(X)
}
\end{displaymath}
o\`u $\Delta_X$ est la diagonale $X\by{\Delta_X} X\times X$. Du diagramme ci-dessus et de la proposition \ref{formule-res}, on d\'eduit:
\[\partial^H(xy)=\partial^F(x)s^G(y)+s^F(x)\partial^G(y)+\{-1\}\partial^F(x)\partial^G(y).\]
\end{proof}

\begin{cor}
La formule \eqref{eq-res-cup} s'applique aussi aux r\'esidus $\partial_m$ de \eqref{residu} et $\partial_{D,g}$ de la d\'efinition \ref{residuBG}.
\end{cor}

\begin{proof} Cela r\'esulte imm\'ediatement de la d\'efinition de ces r\'esidus.
\end{proof}

\begin{rem}
La formule \eqref{eq-res-cup} est analogue \`a celle de Rost \cite[\textbf{P3}, p. 331]{rost} (Un signe appara\^it en plus chez Rost parce que les modules de cycles sont gradu\'es).
\end{rem}
\section{Classes non ramifi\'ees sur un espace classifiant}\label{s9}


\subsection{Le foncteur $A^0_{\NR}(-,M_n)$} On se donne un module de cycles $M=(M_n)$.

\begin{defn}\label{gp-nr-gm} Soit $G\in \Grp$. On d\'efinit:  
\[A^0_{\NR}(BG,M_n)=\bigcap_{(D\subset G,\ g:\mu_m\to Z_G(D))}\Ker(A^0(BG,M_n)\by{\partial_{D,g}} A^0(BD,M_{n-1})),\]
o\`u $\partial_{D,g}$ est comme dans la d\'efinition \ref{residuBG} et $m$ parcourt les entiers $\ge 1$.
\end{defn}

\begin{prop}\label{foncteurNR}
La loi  $G\mapsto A^0_{\NR}(BG,M_n)$ d\'efinit un sous-foncteur de $G\mapsto A^0(BG,M_n)$ sur $\Grp$.
\end{prop}

\begin{proof}
Soit $f:G\to H$ un homomorphisme de groupes: il s'agit de voir que
$f^*A^0_{\NR}(BH,M_n)\subset A^0_{\NR}(BG,M_n)$.

Si $D\subset G$ est un sous-groupe de $G$, posons $D_H=f(D)$: on a $f(Z_G(D))\subset  Z_{H}(D_H)$. Soit $g:\mu_m \to Z_G(D)$. La composition 
\[g_H: \mu_m \by{g} Z_G(D) \by{f} Z_H(D_H)\]
donne un diagramme commutatif
\[
\xymatrix{
A^0(BG,M_n) \ar[r] & A^0(BD\times B\mu_m,M_n)\ar[r]& A^0(BD\times \G_m)\ar[r]& A^0(BD,M_{n-1})\\
A^0(BH,M_n)\ar[r]\ar[u]&  A^0(BD_H\times B\mu_m,M_n)\ar[r]\ar[u]& A^0(BD_H\times \G_m)\ar[r]\ar[u]&A^0(BD_H,M_{n-1})\ar[u]
}
\]
soit
\[\begin{CD}
A^0(BG,M_n) @>\partial^G_{D,g}>> A^0(BD,M_{n-1})\\
@AAA @AAA\\
A^0(BH,M_n) @>\partial^H_{D_H,g_H}>> A^0(BD_H,M_{n-1})
\end{CD}\]
o\`u l'on a \'ecrit $\partial^G,\partial^H$ pour garder la trace de $G,H$. D'o\`u l'inclusion cherch\'ee.
\end{proof}

\subsection{Une formule simplifi\'ee pour $A^0_{\NR}(BG,M_n)$}

\begin{lemme}\label{simplifieG}
Soient $D'\subset D\subset G$ et $g:\mu_m\to Z_G(D)$; notons $g':\mu_m\to Z_G(D) \inj Z_G(D')$. Alors $\Ker \partial_{D,g}\subset \Ker \partial_{D',g'}$.
\end{lemme}

\begin{proof}
Cela r\'esulte du diagramme commutatif suivant
\begin{displaymath}
\xymatrix{
F(BG) \ar[r] \ar[d]^{=}& F(BD\times BI)\ar[r]\ar[d]& F(BD\times \G_m) \ar[r]\ar[d]& F_{-1}(BD)\ar[d]\\
F(BG) \ar[r] & F(BD'\times BI)\ar[r]& F(BD'\times \G_m) \ar[r]& F_{-1}(BD').
}
\end{displaymath}
\end{proof}

Du lemme \ref{simplifieG}, on d\'eduit:

\begin{prop}\label{gpnr-simple}
On a
\[A^0_{\NR}(BG,M_n)=\bigcap_{g:\ \mu_m \to G} \Ker(\partial_{Z_G(g),g})\]
o\`u $Z_G(g)$ est le centralisateur de $g(\mu_m)\subset G$.
\end{prop}

\subsection{Une majoration de la cohomologie non ramifi\'ee} Rappelons de \cite[\S 6]{BG1} le groupe $A^0_{\nr}(BG,M_n) \subset A^0(BG,M_n) $ des classes non ramifiées.

\begin{prop}\label{prop-cohnr}
On a l'inclusion
\[A^0_{\nr}(BG,M_n) \subset A^0_{\NR}(BG,M_n).\]
\end{prop}

\begin{proof} Choisissons une repr\'esentation tr\`es fid\`ele $W$ de $G$ \cite[déf. 2.7]{BG1}. 
Soit $x$ un \'el\'ement de $A^0_{\nr}(k(W)^G,M_n)\subset A^0(BG,M_n)$ \ie $\partial_A(x)=0$ pour tout $A\in \sP(k(W)^G/k)$. Rempla\c cons la notation $\partial_{D,g}$ (d\'ef. \ref{residuBG}) par $\partial^G_{D,g}$, pour garder la trace du groupe $G$. On veut montrer que pour tout le couple $(D,g)$, on a $\partial^G_{D,g}(x)=0$. 

Pour simplifier, posons $I=\mu_m$. Soit $\phi:D\times I \to G$ le morphisme d\'efini par $\phi(d,i)=d.g(i)$. D'apr\`es la d\'efinition des r\'esidus g\'eom\'etriques (d\'ef. \ref{residuBG}), $\partial^G_{D,g}$ se factorise par $\partial^{D\times I}_{D,g}$; autrement dit, on a un  diagramme commutatif:
\[
\xymatrix{
A^0(BG,M_n) \ar[d]^{\phi^*}\ar[dr]^{\partial^G_{D,g}}\\
A^0(BD\times BI,M_n)\ar[r]^{\partial^{D\times I}_{D,g}} & A^0(BD,M_{n-1}).
}
\]

Mais  $\phi^*A^0_\nr(BG,M_n)\subset A^0_\nr(BD\times BI,M_n)$ \cite[\S 6.4]{BG1}. On peut donc supposer que $G=D\times I$. 

On raisonne comme Peyre dans \cite[p. 207]{peyre1}:  choisissons $W$ de la forme $W'\times W_\chi$, o\`u $W'$ est une repr\'esentation tr\`es fid\`ele de $D$ et $W_\chi$ est la repr\'esentation fid\`ele canonique de dimension 1 de $I$. Le diagramme commutatif du  lemme \ref{residuP}:
\[
\xymatrix{
A^0(BD \times BI,M_n) \ar[r]^{\partial^{D\times I}_{D,g}}\ar @{^(->}[d]& A^0(BD, M_{n-1})\ar @{^(->}[d]\\
M_n(k(W'\oplus W_\chi)^{D\times I}) \ar[r]^{\partial_{A_{\chi}}\qquad }& M_{n-1}(k(W')^D)=M_{n-1}(\kappa_{A_{\chi}})
}
\]
o\`u $A_{\chi}\in \sP(k(W'\oplus W_\chi)^{D\times I}/k)$, montre imm\'ediatement que $\partial^{D\times I}_{D,g}A^0_\nr(BD\times BI,M_n)=0$. 
\end{proof}

\begin{rem} Pour $G$ quelconque, il est tr\`es improbable qu'on ait \'egalit\'e dans la proposition \ref{prop-cohnr}: les r\'esidus g\'eom\'etriques $\partial_{D,g}$ ne semblent pas suffisants. Le r\'esultat principal de cet article est qu'on a \'egalit\'e quand $G$ est fini constant et que $k$ contient assez de racines de l'unit\'e. C'est l'objet de la section suivante.
\end{rem}

\section{Le th\'eor\`eme principal}\label{th.princ}
\`A partir de maintenant, on suppose que $G$ est un groupe fini constant d'exposant $m$, o\`u $m$ est un entier inversible dans $k$, et que   $\mu_m\subset k$. 
Le but de cette section est de d\'emontrer le th\'eor\`eme suivant:

\begin{thm}\label{thmprincipal} Sous les hypoth\`esess ci-dessus, on a:
\[A^0_{\NR}(BG,M_n)=A^0_{\nr}(BG,M_n)\]
(\'egalit\'e dans $A^0(BG,M_n)$).
\end{thm}


\subsection{Un sous-groupe interm\'ediaire} 

D'apr\`es la proposition \ref{prop-cohnr}, il suffit de d\'emontrer que $A^0_{\NR}(BG,M_n)\subset A^0_{\nr}(BG,M_n)$. Pour cela, on va d\'efinir un autre sous-groupe de $A^0(BG,M_n)$ contenant $A^0_{\NR}(BG,M_n)$ de la mani\`ere suivante.

 Soit $W$ une repr\'esentation lin\'eaire fid\`ele de $G$. Soient $A\in \sP(k(W)^G/k)$ et $B\in \sP(k(W)/k)$ au dessus de $A$. Soient $D$ le groupe de d\'ecomposition de $B$ dans $G$ et $I$ son groupe d'inertie.

\begin{rem}\label{remI}
Comme l'exposant de $G$ divise $m$, l'ordre de $G$ divise une puissance de $m$ et donc $|G|$ est inversible dans $k$. D'apr\`es \cite[IV, \S 2, cor. 2, cor. 3]{serre}, $I$ est cyclique, canoniquement isomorphe \`a $\mu_q$ avec $q|m$, et \emph{central} dans $D$.
\end{rem}

\begin{defn}\label{gp-nr-sp}
On d\'efinit:  
\[A^0_{\NR,sp}(k(W)^G,M_n)=\bigcap_{(D\subset G,\ g:I\to Z_G(D))}\Ker(A^0(BG,M_n)\by{\partial_{D,g}} A^0(BD,M_{n-1})),\]
o\`u l'intersection porte sur l'ensemble des sous groupes $D,I$ relatifs \`a $A\in \sP(k(W)^G/k)$ comme ci-dessus et $\partial_{D,g}$ est comme dans la d\'efinition \ref{residuBG}.
\end{defn}

Il est clair que
\[A^0_{\NR}(BG,M_n) \subset A^0_{\NR,sp}(k(W)^G,M_n).\]
On va montrer dans les num\'eros suivants:

\begin{thm} \label{cohnr-gm}
Soient $B,A$ comme ci-dessus, et soient $D$ et $I$ les sous-groupes de d\'ecomposition et d'inertie de $B$. Consid\'erons les r\'esidus:
\begin{align*}
A^0(BG,M_n) &\by{\partial_{D,g}}  A^0(BD,M_{n-1}) \ \text{(\cf d\'ef. \ref{residuBG})}\\
\intertext{et}
M_n(k(W)^G) &\by{\partial_A} M_{n-1}(\kappa_A)
\end{align*}
o\`u $\kappa_A$ est le corps r\'esiduel de $A$. Si $x\in A^0(BG,M_n)$ est tel que $\partial_{D,g}(x)=0$, alors $\partial_A(x)=0$. Par cons\'equent,
\[A^0_{\NR,sp}(k(W)^G,M_n) \subset A^0_{\nr}(BG,M_n)\ \text{(\cf d\'ef. \ref{gp-nr-sp} et \cite[d\'ef. 6.1]{BG1})}.\]
\end{thm}

D'o\`u on d\'eduit le th\'eor\`eme principal \ref{thmprincipal} et un peu plus pr\'ecis\'ement: 
\begin{cor}\label{egalite}
\[A^0_{\NR}(BG,M_n) = A^0_{\NR,sp}(k(W)^G,M_n)=A^0_{\nr}(BG,M_n).\]
\end{cor} 

De plus, comme $A^0_{\nr}(BG,M_n)$ et $A^0_{\NR}(BG,M_n)$ sont des foncteurs contravariants en $G$ (\cf \cite[\S 5.5]{BG1} et proposition \ref{foncteurNR}), on obtient:

\begin{cor}
On a aussi un r\'esultat \'equivalent au th\'eor\`eme \ref{thmprincipal} pour la partie r\'eduite (\cf déf. \ref{tildeF(BG)}):
\[\tilde{A}^0_{\NR}(BG,M_n)=\tilde{A}^0_{\nr}(BG,M_n).\]
\end{cor}

\subsection{Lemmes}

Dans cette partie, on garde les hypoth\`eses de la d\'efinition \ref{gp-nr-gm} et les notations pr\'ec\'edentes. 

Pour montrer le th\'eor\`eme \ref{cohnr-gm}, on a besoin des lemmes suivants. Les deux premiers reformulent certains r\'esultats de Saltman \cite{saltman}:

\begin{lemme}\label{ext-corps2}
Soit $k$ un corps contenant le groupe $\mu_q$ des racines de l'unit\'e o\`u $q$ est inversible dans $k$. Soient $G$ un groupe fini et $N$ un groupe cyclique d'ordre $q$. Soit $G'$ une extension centrale de $G$ par $N$. Donnons-nous un $2$-cocycle normalis\'e $c$ de $G$ \`a valeurs dans $N$ d\'efinissant l'extension $G'$, associ\'e \`a une section ensembliste $s$ de la projection $\pi:G'\to G$ (v\'erifiant $s(1)=1$).\\
Soit $k'/k$ une extension de groupe de Galois $G$. Soit $\chi:N \iso \mu_q$ un caract\`ere fid\`ele de $N$ sur $k$. Notons $W_{\chi}$ la $k$-repr\'esentation de dimension un correspondante. Soit $W$ la repr\'esentation de $G'$ induite de $W_{\chi}$. Alors $W$ est fid\`ele et on a un isomorphisme
\[k'(W)/k'(W)^{G'}\iso \Frac(S'(k'/k))/\Frac(R(k'/k))\]
o\`u $G'$ op\`ere sur $k'(W)$ par son action sur $k'$ (via $G$) et sur $W$, et $S'(k'/k),R(k'/k)$ sont associ\'ees \`a  $c$ comme dans Saltman \cite[p. 74, 75]{saltman}. 
\end{lemme}

Rappelons d'abord la d\'efinition de l'extension $S'(k'/k)/R(k'/k)$ de \cite[pp. 74, 75]{saltman}. Dans la situation de l'\'enonc\'e:
\begin{displaymath}
\xymatrix{
1 \ar[r]& N \ar[r]& G'\ar[r]_{\pi} &G \ar[r] \ar @/_/[l]_{s}& 1
}
\end{displaymath}
le $2$-cocycle $c$ est d\'efini par la relation:
\[s(g)s(h)=s(gh)c(g,h). \]

Soit $g'\in G'$: on \'ecrit $g'=n(g')s\pi(g')$ avec $n(g')\in N\subset Z(G')$.

\begin{itemize}
	\item $S''(k'/k):=k'[y(g)\mid g\in G](1/s)$ o\`u $s=\prod_{g\in G} y(g)$. Et $G$ op\`ere sur $S''(k'/k)$ par son action sur $k'$ et par
	\[gy(h)=y(gh)\ \forall g\in G.\]
	\item  $S(k'/k):=S''(k'/k)[x(g)\mid 1\neq g \in G]/(x(g)^q=y(g)/y(1))$. 
	Notons $x(1)=1$. Alors l'action de $G$ sur $S(k'/k)$ \'etend celle sur $S''(k'/k)$ via
	\[gx(h)=[x(gh)/x(h)]\chi(c(g,h)).\]
	\item  $S'(k'/k):=S(k'/k)[\gamma]/(\gamma^q=y(1))$. Donc $N$ op\`ere sur $S'(k'/k)$ par l'action triviale sur $S(k'/k)$ et par
	\[n\gamma=\chi(n)\gamma.\]
	Alors $G'$ op\`ere sur $S'(k'/k)$ via $N$ et $G$.
	\item  $R(k'/k):=S(k'/k)^G$.
\end{itemize}
 
Le diagramme suivant r\'esume ce qui pr\'ec\`ede:
	\begin{displaymath}
\xymatrix{
&& S'(k'/k)\\
k' \ar[r]&S''(k'/k) \ar[r]& S(k'/k)\ar[u]_N\\
k\ar[u]_{G}\ar[rr]&& R(k'/k).\ar[u]_G
}
\end{displaymath}
 

\begin{proof}[D\'emonstration du lemme \protect{\ref{ext-corps2}}]
D'apr\`es la d\'efinition des repr\'esentations induites, on a \[W=k[G']\otimes_{k[N]} W_{\chi}\] de base $\{s(g)\otimes w|g\in G\}$ o\`u $w$ est une base de $W_{\chi}$. L'action de $G'$ sur $W$ est: 
\begin{align}
g'\in G',\ 	g'(s(g)\otimes w)&= g's(g) \otimes w \nonumber\\
	                 &= n(g')s\pi(g')s(g) \otimes w \nonumber\\
	                 &= n(g')s(\pi(g')g)c(\pi(g'),g) \otimes w \nonumber \\
	                 &= s(\pi(g')g) \otimes n(g')c(\pi(g'),g).w \nonumber \\
	                 &= s(\pi(g')g) \otimes \chi(n(g')c(\pi(g'),g))w \nonumber \\
	                 &= \chi(n(g')c(\pi(g'),g)) s(\pi(g')g) \otimes w. \nonumber
\end{align}
En particulier:
\begin{enumerate}
	\item [-]Si $n\in N$, on a:
\[n(s(g) \otimes w)=\chi(n)s(g)\otimes w.\]

  \item [-]Si $h\in G$, on a:
  \[s(h)(s(g)\otimes w)=\chi(c(h,g))s(hg)\otimes w.\]
\end{enumerate}

Montrons que $\Ind^{G'}_{N}\chi$ est une repr\'esentation fid\`ele de $G'$. Pour $g'\in G'$ et $x=\sum_{g\in G}\lambda_g s(g)\otimes w \in W,\ \lambda_g\in k$, on a:
\begin{align}
	g'x=x \iff &g'\sum_{g\in G}\lambda_g s(g)\otimes w=\sum_{g\in G}\lambda_g s(g)\otimes w\nonumber\\
	\iff &\sum_{g\in G}\lambda_g g's(g)\otimes w=\sum_{g\in G}\lambda_g s(g)\otimes w\nonumber\\
	\iff &\sum_{g\in G}\lambda_g \chi(n(g')c(\pi(g'),g)) s(\pi(g')g)\otimes w=\sum_{g\in G}\lambda_{g} s(g)\otimes w\nonumber\\
	\iff &\sum_{g\in G}\lambda_g \chi(n(g')c(\pi(g'),g)) s(\pi(g')g)\otimes w=\sum_{g\in G}\lambda_{\pi(g')g} s(\pi(g')g)\otimes w\nonumber\\
	\iff &\lambda_g \chi(n(g')c(\pi(g'),g)) =\lambda_{\pi(g')g} \ \forall g\in G.\nonumber
\end{align}
Si $g'=n(g')\in N-\{1\}$, alors $\chi(n(g'))\neq 1$ et donc 
\[g'x=x \Leftrightarrow \lambda_g \chi(n(g')) =\lambda_{g} \ \forall g\in G \Leftrightarrow \lambda_g=0\ \forall g\in G\Leftrightarrow x=0.\]
Si $g'=n(g')s\pi(g')\in G'-N$, \ie $\pi(g')\neq 1$, soit $m$ le nombre d'\'el\'ements de $G$, on a:
\[g'x=x \Leftrightarrow \lambda_g \chi(n(g')c(\pi(g'),g)) =\lambda_{\pi(g')g} \ \forall g\in G.   \]
Ici on a $m$ variables $\lambda_g$ mais il y a au maximum $m-1$ \'equations ind\'ependantes. Donc on peut trouver des $\lambda_g$ hors des solutions \ie des $\lambda_g$ tels que $g'x\neq x$.
 
Posons \[\gamma=1\otimes w\ \text{et}\ x(g)=\dfrac{s(g)\otimes w}{1\otimes w}\in k(W)\ \forall g\neq 1.\]

Si $h\in G$, on a
\begin{align}
 s(h)x(g)&=\dfrac{s(h)s(g)\otimes w}{s(h) \otimes w}\nonumber\\	
                 &=\dfrac{\chi(c(h,g))s(hg) \otimes w}{s(h)\otimes w} \nonumber\\
                 &=\chi(c(h,g))x(hg)/x(h). \nonumber
\end{align}

Si $n\in N$, on a
\begin{align}                 
 nx(g) &=\dfrac{\chi(nc(1,g))s(g)\otimes w}{\chi(nc(1,1))1 \otimes w} \nonumber\\
               &=x(g)\ \text{car}\ c(1,1)=c(1,g)=c(g,1)=1.\nonumber                 
\end{align}
Ainsi, $N$ op\`ere trivialement sur $\{x(g)|g\in G\}$. \\

Pour $n\in N$, $n\gamma  =\chi(n)\gamma$.

Pour $h\in G$,
\begin{align}
	 s(h)\gamma &=s(h)\otimes w \nonumber\\
	                    &=(1\otimes w)\dfrac{s(h)\otimes w}{1\otimes w}=\gamma x(h). \nonumber 
\end{align}

Ainsi, on a les m\^emes g\'en\'erateurs et relations que chez Saltman, et donc
\[k'(W)/k'(W)^{G'}\iso \Frac(S'(k'/k))/\Frac(R(k'/k)).\]
\end{proof}

\begin{rem}
Dans la d\'emonstration ci-dessus, si on note $W^{g'}=\{x\in W|g'x=x\}$ (pour $g'\in G'-N$), alors $\dim W^{g'}=[G:\left\langle \pi(g')\right\rangle]$. Donc \[\nu(W)=1\Leftrightarrow \exists g':|G|-[G:\left\langle \pi(g')\right\rangle]=1 \Leftrightarrow |G|=|\left\langle \pi(g')\right\rangle|=2\]
(\cf \cite[rem. 2.8]{BG1}).
\end{rem}

\begin{lemme}\label{ext-corps} Avec les notations du lemme pr\'ec\'edent, on a un $k$-isomorphisme
\[k'(W)^{G'}\simeq k(\sA)(x)\]
o\`u $x$ est une ind\'etermin\'ee, $\sA$ est la $k$-alg\`ebre centrale simple d\'efinie par le $2$-cocycle $c$ et $k(\sA)$ est son corps de d\'eploiement g\'en\'erique (corps des fonctions de la vari\'et\'e de Severi-Brauer de $\sA$).
\end{lemme}

\begin{proof} Cela r\'esulte du lemme \ref{ext-corps2} et de \cite[thm. 1.5]{saltman}.
\end{proof}

\begin{lemme}\label{brauer-propriete}
Soit $K$ un corps complet pour une valuation discr\`ete $v$ de rang un. Soit $A$ l'anneau de valuation de $v$. On suppose qu'on est dans la situation suivante:
\[K\by{G} K_\nr \by{N} K'\]
o\`u $K'/K$ est galoisienne de groupe $G'$, d'inertie $N$. Soit $B\subset K_\nr$ la clôture intégrale de $A$. Soient $\kappa_A,\kappa_B$ des corps r\'esiduels de $A,B$. Supposons que $q$ soit inversible dans $\kappa_A$. Alors\\
a) L'image de $[G']\in H^2(G,N)\allowbreak=H^2(\kappa_B/\kappa_A,N)$ est triviale dans le groupe de Brauer $\Br(\kappa_A)$.\\
b) L'extension $\kappa_B(W)^{G'}/\kappa_A$ est rationnelle, o\`u $W$ est la repr\'esentation du lemme \ref{ext-corps2}.
\end{lemme}

\begin{proof}
a) Consid\'erons le diagramme suivant:
\begin{displaymath}
\xymatrix{
[G']\in  H^2(G,N) \ar[r] \ar[d] & H^2(G,\kappa_B^*) \ar @{^{(}->}[r]\ar @{-->}[dl]^{\phi} & \Br(\kappa_A)\\
 H^2(G,K^*_\nr) \ar @{^{(}->}[d]\\
 \Br(K).
}
\end{displaymath}
Il existe un morphisme $\phi:H^2(G,\kappa_B^*)\to H^2(G,K_\nr^*)$ faisant commuter le triangle, et il est injectif (\cf \cite[pp. 192--194]{serre}). En effet, on a la suite exacte:
\[1 \to U_{K_\nr} \to K_\nr^* \by{v} \Z \to 0\]
scind\'ee par le choix d'une uniformisante de $K$. On en d\'eduit une suite exacte scind\'ee:
\[0\to H^2(G,U_{K_\nr}) \to H^2(G,K_\nr^*) \to H^2(G,\Z) \to 0.\]
De plus, on a une autre suite exacte:
\[1 \to U_{K_\nr}^1 \to U_{K_\nr} \to \kappa_B^* \to 1\]
o\`u $U_{K_\nr}^1$ est le sous-groupe de $U_{K_\nr}$ form\'e des $a\in U_{K_\nr}$ tels que $v(1-a)\geq 1$. On a $H^q(G,U^1_{K_\nr})=0$ pour tout $q\geq 1$ \cite[lemme 2, p. 193]{serre}. Donc 
\[H^2(G,\kappa_B^*) \stackrel{\sim}{\leftarrow} H^2(G,U_{K_\nr}) \inj H^2(G,K_\nr^*).\]

Comme $K'/K$ est une extension de groupe de Galois $G'$ induite par $K_\nr/K$ de groupe $G$ (``the embedding problem"), d'apr\`es \cite[prop. 1.1]{saltman}, l'image de $[G']$ dans $H^2(G,K_\nr^*)$ est triviale. Elle est donc triviale dans $H^2(G,\kappa_B^*)\inj \Br(\kappa_A)$.

b) Cela r\'esulte de a) et du lemme \ref{ext-corps}, puisqu'avec les notations de ce lemme on a $[\sA]=[G']\in \Br(\kappa_A)$.
\end{proof}

\begin{lemme}[``Lemme sans nom tordu''] \label{sans-nom-tordu} 
Soient $G,N,G'$ com\-me dans le lemme \ref{ext-corps}. Soit $k'/k$ une extension de groupe de Galois $G$. Soient $W,W'$ deux repr\'esentations fid\`eles de $G'$ sur $k$.  Alors $k'(W)^{G'}$ et $k'(W')^{G'}$ sont stablement \'equivalents sur $k$. Si l'un est pur, l'autre est stablement pur.
\end{lemme}

\begin{proof}
Soit $U$ un ouvert de $W$ tel que $U$ soit un $G'$-torseur (\cf \cite[rem. 2.8]{BG1}). Alors $U_{k'}=U\times_k \Spec(k')$ est encore un $G'$-torseur (puisque $\Spec(k')$ est un $k$-sch\'ema affine!  \cite[VIII, cor. 7.9]{SGA1} ou \cite[chap. 1, th. 2.23]{milne}). 
 Donc $U_{k'}\times_{k} W'/G'$ est un fibr\'e vectoriel sur $U_{k'}/G'$. Et donc $k'(W\oplus W')^{G'}=k'(U_{k'}\times W')^{G'}$ est transcendant pur sur $k'(W)^{G'}$. On raisonne de m\^eme avec $W'$.
 \end{proof}

\subsection{D\'emonstration du th\'eor\`eme \ref{cohnr-gm}}

D'apr\`es la remarque \ref{remI}, $I$ est cyclique ($I\iso \mu_q$ o\`u $q|m$) et central dans $D$. 
Rappelons aussi que
\[I=\Ker (D \to Gal(\kappa_B/\kappa_A)).\]
 Soit $W'$ une $k$-repr\'esentation fid\`ele de $D$, qui est somme directe d'au moins deux repr\'esentations r\'eguli\`eres de $D$ (\cf \cite[rem. 2.8 et \S 5.5]{BG1}). Soit \[\chi:I\iso \mu_q\inj k^*\] un caract\`ere fid\`ele de $I$ (\cf \cite[d\'em. prop. 3, p. 207]{peyre1}). Soit $\pi$ une uniformisante de $B$ telle que $\pi^q \in k(W)^I$ (un tel $\pi$ existe, \cf \cite[chap. II, prop. 12]{lang}). Alors, $k(W)=k(W)^I[\pi]$ et
\begin{equation}\label{uniformisante}
\sigma \pi = \chi(\sigma)\pi
\end{equation}
pour $\sigma\in I$, parce que $\mu_q\subset k(W)^I$ et donc l'extension $k(W)/k(W)^I$ est kummerienne. 
 
On note $\phi:D\times I \to G$ le morphisme d\'efini par $(d,i)\mapsto di$. Alors $D\times I$ op\`ere sur $W$ via $\phi$.
 
Consid\'erons le diagramme suivant:
\begin{displaymath}
\xymatrix{
\overline{k(W)} \ar[r] & \overline{k(W\oplus W'\oplus \chi)} & \overline{k(W'\oplus \chi)}\ar[l]\\
k(W) \ar[r] \ar[u]& k(W\oplus W'\oplus \chi)\ar[u]& k(W'\oplus \chi) \ar[l]\ar[u]\\
k(W)^G \ar[r] \ar[u]_{G}& k(W\oplus W'\oplus \chi)^{D\times I} \ar[u]_{D\times I}& k(W'\oplus \chi)^{D\times I}\ar[l] \ar[u]_{D\times I}
}
\end{displaymath}
o\`u $D\times I$ op\`ere sur $W\oplus W'\oplus \chi$ par l'action de $D\times I$ sur $W$, l'action de $D$ sur $W'$ et l'action de $I$ sur $\chi$. 

Soit $(X_1,\dots, X_s)$ une base de $W'$. D'apr\`es \cite[lemme 1, p. 156]{bourbaki}, il existe une unique valuation discr\`ete de rang un $w_1$ prolongeant $v_B$ dans $k(W)(X_1)$, telle que 
\[w_1(P)=w_1(\sum_ja_jX_1^j)=\inf_j \{v_B(a_j)+j\xi \}\] 
o\`u $a_j\in k(W)$ et $\xi \in \Z$. On choisit $\xi=0$ et donc $w_1(X_1)=0$. 
De plus, d'apr\`es \cite[prop. 2, p. 157]{bourbaki}, le corps r\'esiduel $\kappa_{w_1}$ de $w_1$ est transcendant pur sur $\kappa_{B}=\kappa_{v_B}$ et plus pr\'ecis\'ement $\kappa_{w_1}=\kappa_{B}(t_1)$ o\`u $t_1$ est l'image de $X_1$ dans $\kappa_{w_1}$. Par r\'ecurrence, il existe une unique valuation discr\`ete $w$ prolongeant $v_B$ dans $k(W\oplus W')$ telle que 
\begin{gather*}
w(\sum_J a_JX^J)=\inf_J\{v_B(a_J)\}\\
\intertext{o\`u}
J=(j_1,\dots,j_s),X^J=X_1^{j_1}\dots X_s^{j_s},a_J\in k(W).
\end{gather*}

Donc $w(X_i)=0,i=1,\dots,s$ pour $X_i\in W'$ et $\kappa_w=\kappa_B(t_1,\dots,t_s)$ o\`u $t_i$ est l'image de $X_i$ dans $\kappa_w$ pour tout $i$. Notons la formule:
\begin{equation} \label{equation1}
w(\lambda_1X_1+\dots +\lambda_sX_s)=0 \quad \text{si}\ k^s\ni (\lambda_1,\dots,\lambda_s)\neq (0,\dots,0).
\end{equation}
Enfin, il existe une unique valuation discr\`ete de rang un $v_{B'}$ prolongeant $w$ donc prolongeant $v_B$ dans $k(W\oplus W'\oplus \chi)=k(W\oplus W')(T)$ telle que 
\[v_{B'}(\sum_{J,l} a_{J,l}X^{J}T^l)=\inf_l\{w(\sum_{J}a_{J,l}X^{J})+l\}=\inf_{J,l}\{v_B(a_{J,l})+l\},\ a_{J,l}\in k(W).\]

Donc $v_{B'}(T)=1$ (on choisit $\xi=1$). Et donc $v_{B'}(T/\pi)=0$ o\`u $\pi$ engendre l'id\'eal maximal $m_B$ de $B$. D'o\`u $\kappa_{B'}=\kappa_{v_{B'}}=\kappa_B(t_1,\dots,t_s,t)$ o\`u $t$ est l'image de $T/\pi$ dans $\kappa_{B'}$. Ainsi l'anneau $B'$ de $v_{B'}$ prolonge $B$ dans $k(W\oplus W'\oplus \chi)$. Posons $A'=B'\cap k(W\oplus W'\oplus \chi)^{D\times I}$. 

Comme $W'$ est une somme de repr\'esentations r\'eguli\`eres de $D$, on peut choisir les $X_i$ ci-dessus de telle sorte qu'ils soient permut\'es par $D$, soit $gX_i=X_{g(i)}$ pour tout $i$. C'est ce que nous faisons dans le lemme qui suit.

\enlargethispage*{20pt}

\begin{lemme}\label{gpdeB'}
Le groupe de d\'ecomposition de $B'$ dans $D\times I$ est $D\times I$ et son groupe d'inertie est $1\times I$. 
\end{lemme}

\begin{proof}
D'abord on va montrer que $(D\times 1)B'= B'$. Soit $g\in D$. Comme $D$ est le groupe de d\'ecomposition de $B$, on a $gB=B$. Par le choix des $X_i$, on a $gX^J=g(X_1^{j_1}\dots X_s^{j_s})=X^{gJ}$. Alors
\begin{align*}
v_{gB'}(\sum_{J,l}a_{J,l}X^JT^l)&=v_{B'}(g(\sum_{J,l}a_{J,l}X^JT^l))=v_{B'}(\sum_{J,l}(ga_{J,l})X^{gJ}T^l) \\
	                                &=\inf_{J,l}\{v_B(ga_{J,l})+l\}=\inf_{J,l}\{v_{gB}(a_{J,l})+l\}\\
	                                &=\inf_{J,l}\{v_B(a_{J,l})+l\}=v_{B'}(\sum_{J,l}a_{J,l}X^JT^l).
\end{align*}
Gr\^ace \`a l'unicit\'e de $B'$, on a $gB'=B'$ pour tout $g\in D$, donc $(D\times 1)B'=B'$. Et on a aussi  $W'\subset B'$. Gr\^ace \`a l'\'equation \eqref{equation1}, on a $v_{B'}(\sum \lambda_i X_i)=0$ avec $\lambda_i \in k$ et donc $W'\inj \kappa_{B'}$. Posons $\overline{W'}=\IM(W')$, c'est le sous-espace vectoriel de $\kappa_{B'}$ de base $t_1,\dots,t_s$ et $D$ op\`ere librement sur $\overline{W'}$. Notons que $W'\to \overline{W'}$ est un ismorphisme $D\times I$-\'equivariant de $k$-espaces vectoriels.

Comme $I\subset D$ et $I=\left\langle \sigma \right\rangle$ op\`ere sur $T$ par $\sigma T=\chi(\sigma)T$ o\`u $\chi(\sigma)\in k^*$, on a de mani\`ere analogue 
\begin{align}
v_{\sigma B'}(\sum_{J,l}a_{J,l}X^JT^l)&=v_{B'}(\sigma(\sum_{J,l}a_{J,l}X^JT^l))=v_{B'}(\sum_{J,l}(\chi(\sigma)(\sigma a_{J,l}))X^{J}T^l)\nonumber \\
	                                &=\inf_{J,l}\{v_B(\chi(\sigma)(\sigma a_{J,l}))+l\}=\inf_{J,l}\{v_{\sigma B}(a_{J,l})+l\}\nonumber\\
	                                &=\inf_{J,l}\{v_B(a_{J,l})+l\}=v_{B'}(\sum_{J,l}a_{J,l}X^JT^l).\nonumber
\end{align}

Ainsi $(1\times I)B'= B'$. En conclusion, $(D\times I)B'=B'$. 
 
D'autre part, $I$ op\`ere trivialement sur $\kappa_B$ \`a cause de la d\'efinition de $I$ et $D$ op\`ere librement sur $\{t_1,\dots,t_s\}$. De plus,
\[\sigma t=\sigma (T/\pi)=\chi(\sigma)T/\chi(\sigma)\pi=t\]
o\`u $\pi$ est choisi comme en \eqref{uniformisante}. Ainsi $1\times I$ op\`ere trivialement sur le corps $\kappa_B(t_1,\dots,t_s,t)$.
\end{proof}

Soient $B_{\chi}$ la restriction de $B'$ \`a $k(W'\oplus \chi)=k(W')(T)$ et $v_{B_{\chi}}$ sa valuation associ\'ee. Alors $v_{B_{\chi}}$ est nulle sur $k(W')$ et $v_{B_{\chi}}(T)=1$ (donc $B_\chi$ est l'anneau que Peyre a consid\'er\'e dans \cite[p. 207]{peyre1}). Posons $A_{\chi}=B_{\chi}\cap k(W'+\chi)^{D\times I}$. On a aussi que le groupe de d\'ecomposition de $B_{\chi}$ est $D\times I$ et son groupe d'inertie est $1\times I$. On a des diagrammes:
\[
\xymatrix{
B \ar[r] & B'& B_{\chi} \ar[l]\\
A \ar[r] \ar[u]& A' \ar[u]& A_{\chi}\ar[l] \ar[u]
}
\qquad
\xymatrix{
\kappa_B \ar[r] & \kappa_{B'}& \kappa_{B_{\chi}} \ar[l]\\
\kappa_A \ar[r] \ar[u]^{D/I}& \kappa_{A'} \ar[u]_{D}& \kappa_{A_{\chi}}.\ar[l] \ar[u]_{D}
}
\]

\begin{lemme}\label{indice-e}
L'indice de ramification  de $A'|A$ est $1$.
\end{lemme}

\begin{proof}
D'abord, d'apr\`es notre construction, l'indice de ramification de $B|A$ est $e_{B|A}=|I|=q$ et celui de $B'|B$ est $e_{B'|B}=v_{B'}(\pi_B)=1$ o\`u $\pi_B$ est une uniformisante de $B$. D'apr\`es le lemme \ref{gpdeB'}, $e_{B'|A'}=|1\times I|=|I|=q$. Enfin, gr\^ace \`a la relation:
\[e_{B'|A}=e_{B'|B}e_{B|A}=e_{B'|A'}e_{A'|A},\]
 on en d\'eduit $e_{A'|A}=1$.
\end{proof}

On en d\'eduit un diagramme commutatif:
\begin{displaymath}
\xymatrix{
A^0(BG,M_n) \ar[r] \ar @{^{(}->}[d]& A^0(BD\times BI,M_n) \ar[rr]^{\partial_{D\times I,g}}\ar @{^{(}->}[d]\ar @{^{(}->}[d]\ar @{^{(}->}[rd]&& A^0(BD,M_{n-1})\ar @{^{(}->}[d]\\
M_n(k(W)^G)\ar[r] \ar[d]^{\partial_A} & M_n(k(W\oplus W'\oplus \chi)^{D\times I})\ar[d]^{\partial_{A'}}& M_n(k(W'\oplus \chi)^{D\times I}) \ar[r]^{\partial_{A_{\chi}}}\ar[d]^{\partial_{A_{\chi}}}\ar[l] & M_{n-1}(k(W')^D) \ar[ld]^{=}\\
 M_{n-1}(\kappa_A)\ar[r]^\alpha & M_{n-1}(\kappa_{A'}) & M_{n-1}(\kappa_{A_{\chi}}).\ar[l]_{e_{A'|A_{\chi}}\cdot }
}
\end{displaymath}

En effet, le rectangle sup\'erieur gauche et le triangle sup\'erieur central commutent  par d\'efinition du foncteur $A^0(-,M_n)$ et des morphismes $A^p(X,M_n) \to A^p(Y,M_n)$ pour $Y\to X$ \cite[\S 12]{rost}. Le trap\`eze sup\'erieur droit commute gr\^ace au lemme \ref{residuP}. 
Enfin, la commutativit\'e des rectangles inf\'erieurs r\'esulte du lemme \ref{indice-e} et de l'axiome (R3a) de \cite{rost}.

Pour conclure, il suffit de montrer que la fl\`eche $\alpha$ est injective,  et pour cela, il suffit de voir que l'extension $\kappa_{A'}/\kappa_A$ est unirationnelle \cite[lemme 5.3]{BG1}. 

 Soit $W''=\Ind^D_I\chi$ (vue comme $k$-repr\'esentation). 
 Comme $\chi$ est un facteur direct de la repr\'esentation r\'eguli\`ere de $I$, $W''$ est un facteur direct de la repr\'esentation r\'eguli\`ere de $D$, donc $W''\subset W'$. D'apr\`es le lemme \ref{ext-corps2}, $W''$ est fid\`ele. On note $\overline{W''}$ l'image de $W''$ dans $\kappa_{B'}$. 

D'apr\`es le lemme \ref{brauer-propriete} b), $\kappa_B(\overline{W''},t)^D=\kappa_B(\overline{W''})^D(t)$ est transcendant pur sur $\kappa_A$. D'apr\`es le lemme \ref{sans-nom-tordu}, $\kappa_{A'}=\kappa_B(\overline{W'},t)^D$ et $\kappa_B(\overline{W''},t)^D$ sont stablement \'equivalents sur $\kappa_A$ et finalement $\kappa_{A'}$ est stablement pur sur $\kappa_A$. 

\begin{rem}
La repr\'esentation $W''$ intervenant \`a la fin de la d\'emonstration ci-dessus provient du lemme \ref{ext-corps}; elle joue un r\^ole cl\'e dans la d\'emonstration. En principe, on aurait pu utiliser $W''\oplus W''$ \`a la place de $W'$ ci-dessus mais cela rendrait la v\'erification du lemme \ref{gpdeB'} plus d\'elicate. Pour cette raison, nous avons pr\'ef\'er\'e proc\'eder indirectement en passant par la repr\'esentation r\'eguli\`ere de $D$.
\end{rem}
\bigskip

\subsection{Un raffinement du th\'eor\`eme principal}

\begin{defn}\label{gpnab}
Soit $G$ $k$-groupe alg\'ebrique lin\'eaire. On d\'efinit
\[A^0_{\nab}(BG,M_n)=\bigcap_A \Ker(A^0(BG,M_n)\to  A^0(BA,M_n))\]
o\`u $A$ parcourt les sous-groupes ab\'eliens ferm\'es de $G$, de type multiplicatif d\'eploy\'e.
\end{defn}

Notons que $A^0_{\nab}(BG,M_n)\subset \tilde A^0(BG,M_n)$ (consid\'erer $A=1$ dans la d\'efinition \ref{gpnab}).

\begin{lemme}\label{l-graceF}
Avec les notations de la d\'efinition \ref{gpnab}, on a
\[\tilde A^0_{\nr}(BG,M_n)\subset A^0_{\nab}(BG,M_n).\]
\end{lemme}

\begin{proof}
C'est \'evident par fonctorialit\'e, puisque $\tilde A^0_{\nr}(BA,M_n)\allowbreak=0$ pour tout $A$ de type multiplicatif d\'eploy\'e \cite[th. 2.29 et cor. 6.5]{BG1}. 
\end{proof}

\begin{lemme}[``Lemme de Bogomolov'']\label{l-genB} Supposons $G$ fini cons\-tant et d'exposant $m$, avec $\mu_m\subset k$. 
Pour tout $D\subset G$ et $g:I=\mu_m \to Z_G(D)$, on a
\[\partial_{D,g}(A^0_{\nab}(BG,M_n))\subset A^0_{\nab}(BD,M_{n-1}).\]
\end{lemme}

\begin{proof}
Soit $A_D$ un sous-groupe ab\'elien de $D$. Posant $A=\left\langle A_D,g(\mu_m)\right\rangle$ (qui est ab\'elien!), on a le diagramme commutatif suivant
\[\begin{CD}
A^0(BG,M_n) @>>> A^0(BA,M_n)\\
@VVV @VVV\\
A^0(BD \times BI,M_n) @>>> A^0(BA_D\times BI,M_n)\\
@VVV @VVV\\
A^0(BD,M_{n-1}) @>>> A^0(BA_D,M_{n-1}).
\end{CD}\]

D'o\`u l'\'enonc\'e.
\end{proof}

\begin{cor}\label{cor-raffine}
On a la suite exacte suivante
\begin{equation}\label{se-raffine}
0\to \tilde A^0_{\nr}(BG,M_n) \to A^0_{\nab}(BG,M_n) \by{\partial_{D,g}} \bigoplus_{D,g} A^0_{\nab}(BD,M_{n-1}).
\end{equation}
\end{cor}

\begin{proof}
D'apr\`es le th\'eor\`eme \ref{thmprincipal}, le lemme \ref{l-genB} et le lemme \ref{l-graceF}, on a le diagramme suivant
\begin{displaymath}
\xymatrix{
0 \ar[r]& \tilde A^0_{\nr}(BG,M_n) \ar @{^{(}->}[dr] \ar[r]& \tilde A^0(BG,M_n)\ar[r] & \bigoplus_{D,g} A^0(BD,M_{n-1})\\
&& A^0_{\nab}(BG,M_n)\ar[r]\ar @{^{(}->}[u]& \bigoplus_{D,g} A^0_{\nab}(BD,M_{n-1})\ar @{^{(}->}[u]
}
\end{displaymath}
o\`u la premi\`ere ligne est une suite exacte. D'o\`u \eqref{se-raffine}.
\end{proof}

Pour le théorème suivant, rappelons la définition \cite[déf. 5.4]{BG1}:

\begin{defn}\label{connectif}
Un module de cycles $M$ est dit \textit{connectif} s'il existe $n_0\in \Z$ tel que $M_n=0$ pour tout $n<n_0$. 
\end{defn}

\begin{thm}\label{c-gpnab}
Avec l'hypoth\`ese de la d\'efinition \ref{gpnab}, soient $M$ un module de cycles et $n\in \Z$.
\begin{thlist}
	\item  Si $A^0_{\nab}(BG,M_n)=0$, alors $\tilde{A}^0_{\nr}(BG,M_n)=0$.
	\item  Si $M$ est connectif et  
	\[\forall H\subset G, \forall m\le n, \tilde{A}^0_{\nr}(BH,M_m)=0,\]
	alors
	\[\forall H\subset G, \forall m\le n, A^0_{\nab}(BH,M_m)=0.\]
\end{thlist}
\end{thm}

\begin{proof}
On a tout de suite (i) gr\^ace au lemme \ref{l-graceF}. On montre (ii) par r\'ecurrence sur $m$. Comme $M$ est connectif, c'est \'evident pour $m$ petit. Supposons que ce soit vrai pour $m-1$. Utilions la suite exacte \eqref{se-raffine} avec $A^0_{\nab}(BD,M_{m-1})=0\ \forall D$, on a
\[A^0_{\nab}(BG,M_m)\osi \tilde A^0_{\nr}(BG,M_m)=0.\]

Et c'est vrai aussi pour tout sous-groupe $H$ de $G$.
\end{proof}

\section{Reformulation: $0$-cycles sur le compactifi\'e de $BG$}\label{s11}



\subsection{Le foncteur $\overline{CH}_0$}

\begin{defn}\label{Chgpcomp}
Soient donn\'es $X\in \Sm(k)$ et une compactification $j:X\inj \bar{X}$ o\`u $\bar{X}$ est projectif et lisse. 
 On d\'efinit
\[\overline{CH}_0(X)=CH_0(\bar{X})=A_0(\bar{X},K^M_0).\]
\end{defn}

\begin{prop}\label{CH_0}
Si $k$ est de caract\'eristique z\'ero, $\overline{CH}_0(X)$ ne d\'e\-pend que de $X$ et $\overline{CH}_0$ d\'efinit un foncteur homotopique et pur en coniveau $\geq 1$.
\end{prop}

La partie d\'elicate de cette proposition est la fonctorialit\'e.

\begin{proof}
Notons $\Sm^\proj$ la sous-cat\'egorie pleine de $\Sm$ form\'ee des sch\'emas projectifs, $S_b$ la classe des morphismes birationnels de $\Sm$. D'apr\`es \cite[thm. 2.1]{KS}, on a un diagramme commutatif de foncteurs
\begin{displaymath}
\xymatrix{
\Sm^\proj\ar[d]\ar @{^{(}->}[r] & \Sm \ar[d]\\
S^{-1}_b\Sm^\proj \ar[r]^{\quad \sim} & S^{-1}_b\Sm
}
\end{displaymath}
o\`u le foncteur du bas est une \'equivalence de cat\'egories.

Consid\'erons le foncteur $F: \Sm \to \Ab$ donn\'e par $F(X)=H_0(X,\Z)$ (homologie de Suslin, \cf \cite[d\'ef. 3.14]{BG1}). Il v\'erifie les hypoth\`eses de \cite[\S 3]{KS} par rapport au digramme ci-dessus. En effet, pour $X$ projectif et lisse, $F(X)=H_0(X,\Z)=CH_0(X)$  et d'apr\`es Fulton \cite[ex. 16.1.11, p .312]{fulton}, $CH_0(X)=CH_0(Y)$ pour tout $X\to Y$ dans $S_b$. Donc on a un isomorphisme naturel de foncteurs:
\[(\Sm^\proj\to \Sm \by{F} \Ab)\simeq (\Sm^\proj\to S_b^{-1}\Sm^\proj\by{CH_0} \Ab).\]


En appliquant \cite[\S 3 et thm. 2.1]{KS}, on obtient une transformation naturelle de foncteurs
\[H_0(X,\Z)\to \overline{CH}_0(X)\]
o\`u $\overline{CH}_0(X):=CH_0(\bar{X})$ pour $\bar{X}$ une compactification lisse de $X$: c'est exactement le foncteur de la d\'efinition \ref{Chgpcomp}, qui est donc bien d\'efini.

De plus, $\overline{CH}_0$ est homotopique et pur en coniveau $\geq 1$. En effet, soit $U\inj X$ une immersion ouverte, alors $\bar{X}$ est aussi un compactifi\'e de $U$:
\[U\inj X \inj \bar{X}.\]
Donc
\[\overline{CH}_0(U)=\overline{CH}_0(\bar{X})=CH_0(X).\]

D'autre part, soit $f:E\to X$ un fibr\'e vectoriel. Soit $U$ un ouvert de $X$, on a le cart\'esien suivant
\begin{displaymath}
\xymatrix{
j^*E \ar @{^{(}->}[r]|o\ar[d]& E \ar[d]^{f}\\
U \ar @{^{(}->}[r]^{j}|o& X
}
\end{displaymath}
o\`u $j^*E\cong U\times \A^n$ quand $U$ est assez petit. Alors 
\begin{displaymath}
\xymatrix{
\overline{CH}_0(j^*E) \ar[r]^{\sim}\ar[d]& \overline{CH}_0(E) \ar[d]\\
\overline{CH}_0(U) \ar[r]^{\sim}& \overline{CH}_0(X).
}
\end{displaymath}
Il nous am\`ene \`a consid\'erer le cas $E=X\times \A^n$.

Soit $j:X\inj Y$ une compactification de $X$, nous avons le diagramme commutatif suivant:
\begin{displaymath}
\xymatrix{
E=X\times \A^n \ar @{^{(}->}[r]\ar[d]& Y\times \A^n \ar @{^{(}->}[r]& Y\times \bP^n \ar[d]\\
X \ar @{^{(}->}[rr]&& Y
}
\end{displaymath}
o\`u $Y\times \bP^n$ est une compactification lisse de $E$. D'apr\`es \cite[III, thm. 3.3]{fulton}, on a $CH_0(Y\times \bP^n)\iso CH_0(Y)$. Donc 
\[\overline{CH}_0(E)=\overline{CH}_0(X).\]
\end{proof}

\begin{rem}\label{remCH_0}
Comme sous-produit de la d\'emonstration de la proposition \ref{CH_0}, on obtient un morphisme de foncteurs
\[H_0(-,\Z) \to \overline{CH}_0.\]
Ce morphisme est surjectif d'apr\`es \cite[prop. 6.1]{kahn0}.

Soit $G\in \Grp$: d'apr\`es la proposition \ref{CH_0} et \cite[d\'ef. 2.14]{BG1}, $\overline{CH}_0(BG)$ est bien d\'efini et on a une surjection
\[H_0(BG,\Z) \surj \overline{CH}_0(BG).\]
\end{rem}

\begin{qn} Peut-on d\'efinir $\overline{CH}_0(-)$ \emph{a priori} comme quotient de $H_0(-,\Z)$, sans utiliser la r\'esolution des singularit\'es (sous-jacente \`a la preuve de la proposition \ref{CH_0})?
Voir remarque \ref{r9.6} pour une r\'eponse dans le cas particulier de $BG$.
\end{qn}

\subsection{Un r\'esultat dual du th\'eor\`eme \ref{thmprincipal}}\label{c-eq}

\begin{prop}\label{prop-dual}
Si $k$ est de caract\'eristique z\'ero, on a la suite exacte suivante:
\begin{equation}\label{eq-dual}
\bigoplus_{D\subset G,g:I\to Z_G(D)} H_{-1}(BD,\Z(-1)) \to H_0(BG,\Z) \to \overline{CH}_0(BG)\to 0. 
\end{equation}
\end{prop}

\begin{proof}
D'apr\`es le th\'eor\`eme \ref{thmprincipal}, on a:
\[0\to A^0_{\nr}(BG,M_0) \to A^0(BG,M_0) \to \bigoplus_{D,g:I\to Z_G(D)} A^0(BD,M_{-1}).\]
Soient $U_G$ un $G$-torseur lin\'eaire de coniveau $\geq c$ et $\bar{X}$ une compactification lisse de $U_G/G$. D'apr\`es \cite[d\'ef. 6.1 et prop. 6.6]{BG1}, on a
\[A^0_{\nr}(BG,M_0)=A^0_{\nr}(U_G/G,M_0)=A^0(\bar{X},M_0).\]
De plus, d'apr\`es \cite[thm. 1.3]{kahn0}, pour $X$ lisse, on a
\[A^0(X,M_0)\cong \Hom_{\CM}(H^X,M) \]
o\`u $\CM$ est la cat\'egorie des modules de cycles et pour tout corps $F/k$,
\[H^X_n(F)=H_{-n}(X_F,\Z(-n)).\]
Donc on a la suite exacte:
\begin{multline*}
0\to \Hom_{\CM}(H^{\bar{X}},M) \to \Hom_{\CM}(H^{U_G/G},M)\\
 \to \bigoplus_{D,g:I\to Z_G(D)} \Hom_{\CM} (H^{U_D/D}[1],M).
\end{multline*}
En appliquant le lemme de Yoneda dans la cat\'egorie ab\'elienne $\CM$, on en tire la suite exacte
\[\bigoplus_{D,g:I\to Z_G(D)}H^{U_D/D}[1]\to H^{U_G/G}\to H^{\bar{X}}\to 0.\]
D'o\`u la suite exacte suivante sur le corps de base $k$:
\[\bigoplus_{D,g:I\to Z_G(D)} H_{-1}(BD,\Z(-1)) \to H_0(BG,\Z) \to H_0(\bar{X},\Z) \to 0\]
o\`u $H_0(\bar{X},\Z)=A_0(\bar{X},K^M_0)=\overline{CH}_0(BG)$ (\cf d\'ef. \ref{Chgpcomp}).

Le morphisme $H_{-1}(BD,\Z(-1)) \to H_0(BG,\Z)$ est donn\'e explicitement par
\[\begin{CD}
	H_{-1}(BD,\Z(-1))&=&\Hom_{\DM}(\Z,M(U_D/D)(1)[1])\\
	                 &&@VVV\\
	                 &&\Hom_{\DM}(\Z,M(U_D/D)\otimes M(\G_m))\\
	                 &&@V{(*)}VV\\
	                 &&\Hom_{\DM}(\Z,M(U_D/D)\otimes M(U_{\mu_m}/\mu_m))\\
	                 &&@VVV\\
	H_0(BG,\Z)       &=&\Hom_{\DM}(\Z,M(U_G/G))\nonumber
\end{CD}\]
o\`u $(*)$ est donn\'e par $\G_m \to B\mu_m$.
\end{proof}

\begin{rem}\label{r9.6}
Si $k$ est de caract\'eristique $p$, la suite exacte \eqref{eq-dual} donne une d\'efinition de ``$\overline{CH}_0(BG)$''.
\end{rem}

\subsection{Conditions \'equivalentes pour avoir des groupes non ramifi\'es triviaux}

\begin{defn}\label{gpnrtrivial}
Soient $k$ un corps et $X$ un sch\'ema lisse sur $k$. On dit que $X$ \textit{n'a pas d'invariants non ramifi\'es}\footnote{de type motivique, \`a cause des invariants dans le groupe de Witt\dots} si pour tout module de cycles $M$, 
\[M_n(k)\iso A^0_{\nr}(X,M_n).\]
\end{defn}

\begin{thm}\label{ceq-gpnrtr}
Si $k$ est de caract\'eristique z\'ero, alors les conditions suivantes sont \'equivalentes:
\begin{enumerate}
	\item [a)] $X$ n'a pas d'invariants non ramifi\'es (d\'ef. \ref{gpnrtrivial});
	\item [b)] L'application $deg: \overline{CH}_0(X_F)\to \Z$ est bijective pour toute extension $F/k$.
\end{enumerate}
\end{thm}

\begin{proof}
D'apr\`es la d\'efinition \ref{Chgpcomp} et la r\'esolution des singularit\'es, on a
\[\overline{CH}_0(X_F)=CH_0(\bar{X}_F)\]
o\`u $\bar{X}$ est une compactification propre et lisse de $X$. Donc ce th\'eor\`eme est exactement 
\cite[th 2.11]{merkurjev}.
\end{proof}

\begin{cor}\label{Chgpcomptr}
Soit $G\in \Grp$. Si $k$ est de caract\'eristique z\'ero, alors les conditions suivantes sont \'equivalentes:
\begin{enumerate}
	\item [a)] ``$BG$'' n'a pas d'invariants non ramifi\'es (d\'ef. \ref{gpnrtrivial});
	\item [b)] $ \widetilde{\overline{CH}}_0(BG_F)=0$ pour toute extension $F/k$ (déf. \ref{tildeF(BG)}).
\end{enumerate}
\end{cor}

\begin{proof}
Remarquons que pour $\Spec k \to G$, on a
\begin{displaymath}
\xymatrix{
\overline{CH}_0(BG_F) \ar[r]& \Z\\
\overline{CH}_0(B1)\ar[u]\ar[ur]_{\sim}
}
\end{displaymath}
D'o\`u le r\'esultat d'apr\`es le th\'eor\`eme \ref{ceq-gpnrtr}.
\end{proof}

\begin{cor}
Soit $G$ un groupe fini d'exposant $m$ sur un corps $k$ contenant $\mu_m$, o\`u $m$ est inversible dans $k$. Alors les conditions suivantes sont \'equivalentes:
\begin{enumerate}
	\item [a)] ``$BG$'' n'a pas d'invariants non ramifi\'es (d\'ef. \ref{gpnrtrivial});
	\item [b)] Pour tout module de cycles $M$ et pour tout $n$,
	\[\tilde{A}^0(BG,M_n)\by{\partial_{D,g}} \bigoplus_{D,g} A^0(BD,M_{n-1})\]
	est injectif.
	\item [c)] Pour toute extension $F/k$,
	\[\bigoplus_{D,g} H_{-1}(BD_F,\Z(-1)) \to \tilde{H}_0(BG_F,\Z)\]
	est surjectif.
\end{enumerate}
\end{cor}

\begin{proof}
$(a)\Leftrightarrow(b)$ gr\^ace au th\'eor\`eme \ref{thmprincipal} et \`a la d\'efinition \ref{gpnrtrivial}. Et $(a)\Leftrightarrow(c)$ \`a cause de la proposition \ref{prop-dual} et du corollaire \ref{Chgpcomptr}.
\end{proof}

\begin{cor}\label{c11.1}
Soit $G$ un groupe fini d'exposant $m$ sur un corps $k$ contenant $\mu_m$, $m$ est inversible dans $k$. Alors les conditions suivantes sont \'equivalentes:
\begin{enumerate}
	\item [$i)$] Pour tout $H\subset G$, ``$BH$'' n'a pas d'invariants non ramifi\'es (d\'ef. \ref{gpnrtrivial});
	\item [$i^{bis})$] Comme i) mais \`a valeurs dans un module de cycles connectif (d\'ef. \ref{connectif}).
	\item [$ii)$] Pour tout $H\subset G$ et pour toute extension $F/k$,
	\[\bigoplus_{A\subset H} H_0(BA_F,\Z) \surj H_0(BH_F,\Z)\]
	o\`u $A$ parcourt les sous-groupes ab\'eliens de $H$.
\end{enumerate}
\end{cor}

\begin{proof}
On va montrer $(i)\Rightarrow (i^{bis})\Rightarrow (ii) \Rightarrow (i)$.

$(i)\Rightarrow (i^{bis})$ est \'evident.

$(i^{bis})\Rightarrow (ii)$: Pour tout $H\subset G$ et pour tout $M$ connectif, d'apr\`es le th\'eor\`eme \ref{c-gpnab}, on a
\[A^0(BH,M_n)\inj \bigoplus_A A^0(BA,M_n).\]
D'apr\`es \cite[th. 5.17]{BG1}, 
ceci implique
\[\Hom_{\CM}(H^{BH},M[n])\inj \bigoplus_A \Hom_{\CM}(H^{BA},M[n])\]
pour tout $M$ connectif. Comme $H^{BH},H^{BA}$ sont connectifs  \cite[lemme 5.15]{BG1}, ceci implique par le lemme de Yoneda
\[\forall H\subset G,  \bigoplus_A	H^{BA} \surj H^{BH}.\]
En \'evaluant $H_0$ sur un corps $F$, on a
\[\forall H\subset G, \forall F/k, \bigoplus_{A\subset H} H_0(BA_F,\Z) \surj H_0(BH_F,\Z).\]

$(ii)\Rightarrow (i)$: Pour tout module de cycles $M$, on a \cite[th. 5.17]{BG1}
\[A^0(BH,M_0)\iso \Hom_{\CM}(H^{BH},M)\iso \Hom_{\HI}(h^{Nis}_0(BH),\mathcal{M}_0).\]
Donc $(ii)$ implique
\[ \bigoplus_{A\subset H} h^{Nis}_0(BA)(\Spec F) \surj h^{Nis}_0(BH)(\Spec F)\]
pour tout $F/k$. Pour conclure, on a besoin du lemme suivant:

\begin{lemme}\label{f-epimor}
Soient $\sF,\sG \in \HI$ deux faisceaux Nisnevich avec transferts invariants par homotopie et $f:\sF \to \sG$. Alors,
\[f\ \text{est un \'epimorphisme}\Leftrightarrow \forall F/k, f_F \ \text{est surjectif}.\]
\end{lemme}

\begin{proof}
C'est \'evident pour $\Rightarrow$. Pour $\Leftarrow$, posons $\sH=\Coker f$, on a $\sH_{\Spec F}=0\ \forall F/k$. D'apr\`es \cite[cor. 11.2]{MVW}, $\sH=0$. Donc $f$ est un \'epimorphisme.
\end{proof}

Appliquons le lemme \ref{f-epimor} pour $\sF=\oplus_A h^{Nis}_0(BA),\ \sG=h^{Nis}_0(BH)$, on a
\[\forall H\subset G,\bigoplus_{A\subset H} h^{Nis}_0(BA) \surj h^{Nis}_0(BH).\]
Ceci implique 
\[\forall H\subset G,A^0(BH,M_n) \inj \bigoplus_A A^0(BA,M_n)\]
pour tout module de cycles $M$. On en d\'eduit $(i)$ gr\^ace au lemme \ref{l-graceF}.
\end{proof}

\section{Application: th\'eor\`emes de Bogomolov et de Peyre}\label{bog.peyre}

On retrouve ici les th\'eor\`emes de Bogomolov \cite[thm. 7.1]{ct-s} et de Peyre \cite[thm. 1]{peyre1}.

\subsection{G\'en\'eralit\'e du th\'eor\`eme de Bogomolov }


\begin{lemme}\label{gpnab1}
Soit $k$ un corps contenant $\mu_m$, o\`u $m$ est inversible dans $k$. Soit $G$ un groupe fini  d'exposant $m$. On a
\[A^0_{\nr}(BG,H^2_{\et}(\Z))=A^0_{\nab}(BG,H^2_{\et}(\Z))=0.\]
\end{lemme}

\begin{proof}
On a
\begin{align*}
&A^0_{\nab}(BG,H^2_{\et}(\Z))\\
&=\Ker(A^0(BG,H^2_{\et}(\Z)) \to \bigoplus_A A^0(BA,H^2_{\et}(\Z)))\ (\text{d\'ef. \ref{gpnab}}) \\
                           &=\Ker(H^2_{\et}(BG,\Z) \to \bigoplus_A H^2_{\et}(BA,\Z))\ (\text{\cite[th. 8.3 a), $n=0$]{BG1}}) \\
                           &=\Ker(H^2(G,\Z)\oplus H^2_{\et}(k,\Z) \to \bigoplus_A H^2(A,\Z) \oplus H^2_{\et}(k,\Z)) \text{\cite[(7.4)]{BG1}}\\
                           &=\Ker(\Hom(G,\Q/\Z)\oplus H^2_{\et}(k,\Z) \to \bigoplus_A \Hom(A,\Q/\Z) \oplus H^2_{\et}(k,\Z))\\
                           &=0
\end{align*}
o\`u $A$ parcourt les sous-groupes ab\'eliens de $G$. On conclut avec le lemme \ref{l-graceF}.
\end{proof}

\begin{thm}\label{thm-B}
 Soit $k$ un corps contenant $\mu_m$, o\`u $m$ est inversible dans $k$. Soient $G$ un groupe fini  d'exposant $m$ et $W$ une $k$-repr\'esentation fid\`ele de $G$. On a
\begin{align}\label{geneq-B}
\widetilde{\Br}_{\nr}(k(W)^G)&=\bigcap_{A\in \sB_G} \Ker (H^2(G,k^*)\to H^2(A,k^*)).
\end{align}
o\`u $\sB_G$ est l'ensemble des sous-groupes bicycliques de $G$ et $\widetilde{\Br}_{\nr}(k(W)^G)$ est la partie r\'eduite de $\Br_{\nr}(k(W)^G)=\Br_{\nr}(BG)$ (\cf déf. \ref{tildeF(BG)}).

En particulier, si $k=k_s$ est s\'eparablement clos, on a \cite{bogomolov} 
\begin{equation}\label{eq-B}
\Br_{\nr}(k_s(W)^G)=\bigcap_{A\in \sB_G} \Ker (H^2(G,\Q/\Z)\to H^2(A,\Q/\Z)).
\end{equation}


En g\'en\'eral ($\mu_m \subset k$),
\[\widetilde{\Br}_{\nr}(k(W)^G) \inj \widetilde{\Br}_{\nr}(k_s(W)^G)=\Br_\nr(k_s(W)^G).\]
\end{thm}

\begin{proof}
Rappelons que 
\[\Br(k(W)^G)=H^2_{\et}(k(W)^G,\Q/\Z(1))=A^0(BG,H^3_{\et}(\Z(1))).\]
Utilisons le corolaire \ref{cor-raffine} et le lemme \ref{gpnab1}: on a
\begin{multline}\label{v-faible}
\Br_{\nr}(BG) =\Br_{\nab}(BG) \\
\subset \bigcap_{A\in \sB_G}\Ker(A^0(BG,H^3_{\et}(\Z(1))) \to A^0(BA,H^3_{\et}(\Z(1)))).
\end{multline}

En fait, on a \'egalit\'e. En effet, soit $\gamma$ appartenant \`a la partie droite de \eqref{v-faible}, on raisonne comme Peyre \cite[rem. 4]{peyre1}. Soient $D\subset G$ et $g:I=\mu_m\to Z_G(D)$. Soit $x\in D$, alors $A=\left\langle x,I\right\rangle$ est un sous-groupe bicyclique de $G$. On a le diagramme commutatif suivant
\begin{displaymath}
\xymatrix{
H^3_{\et}(BG,\Z(1)))\ar[r]^{\partial_{D,g}}\ar[d]& H^2_{\et}(BD,\Z))\ar[d]\ar @{=}[r]& \Hom(D,\Q/\Z)\oplus H^2_{\et}(k,\Z)\ar[d]\\
H^3_{\et}(BA,\Z(1)))\ar[r]^{\partial_{\left\langle x\right\rangle,g}}& H^2_{\et}(B\left\langle x\right\rangle,\Z))\ar @{=}[r]]& \Hom(\left\langle x\right\rangle, \Q/\Z)\oplus H^2_{\et}(k,\Z).
}
\end{displaymath}
où les égalités résultent de \cite[(7.4)]{BG1}.

Comme l'image de $\gamma$ dans $H^3_{\et}(BA,\Z(1))$ et donc dans $H^2_{\et}(B\left\langle x\right\rangle,\Z)$ est nulle pour tout $x\in D$, son image dans $H^2_{\et}(BD,\Z)$ l'est aussi. Alors $\gamma \in \Br_{\nr}(k(W)^G)$. Ainsi, on obtient 
\begin{align*}
\widetilde{\Br}_{\nr}(k(W)^G)&=\bigcap_{A\in \sB_G} \Ker (\tilde{A}^0(BG,H^3_{\et}(\Z(1)))\to \tilde{A}^0(BA,H^3_{\et}(\Z(1))))\\
&=\bigcap_{A\in \sB_G} \Ker (\tilde{H}^3_{\et}(BG,\Z(1))\to \tilde{H}^3_{\et}(BA,\Z(1)))\ (\text{\cite[(8.6)]{BG1}}).
\end{align*}
 
D'apr\`es \cite[(7.5)]{BG1}, on a $\tilde{H}^3_{\et}(BG,\Z(1))=H^2(G,k^*)$. D'o\`u \eqref{geneq-B}.

 Si $k$ est s\'eparablement clos, on a
\[H^3_{\et}(k_s,\Z(1))=H^2(k_s,\G_m)=\Br(k_s)=0, \quad H^2(G,\Q/\Z(1))\iso H^2(G,k^*).\]
 
 D'o\`u \eqref{eq-B}.

En g\'en\'eral, on a $\widetilde{\Br}_{\nr}(k_s(W)^G)=\Br_{\nr}(k_s(W)^G)$ car $\Br(k_s)=0$. Consid\'erons la longue suite exacte suivante:
\[H^1(G,k^*) \to H^1(G,k_s^*) \to H^1(G,k_s^*/k^*) \to H^2(G,k^*) \to H^2(G,K_s^*)\]

Comme $\mu_m \subset k^* \subset k_s^*$, on a
\[H^1(G,k^*)\osi H^1(G,\mu_m) \iso H^1(G,k_s^*).\]

Donc la suite
\[0 \to H^1(G,k_s^*/k^*) \to H^2(G,k^*) \to H^2(G,K_s^*)\]
est exacte. D'o\`u le diagramme commutatif de suites exactes:
\begin{displaymath}
\xymatrix{
&& 0\ar[d] & 0\ar[d]\\
&& \widetilde{\Br}_{\nr}(BG)\ar[r] \ar[d]& \Br_{\nr}(BG_{k_s})\ar[d]\\
0\ar[r] & H^1(G,k_s^*/k^*)\ar[r]\ar @{-->}[d]& H^2(G,k^*) \ar[r]\ar[d] & H^2(G,k_s^*)\ar[d]\\
0\ar[r] &\displaystyle \bigoplus_{A\in \sB_G}H^1(A,k_s^*/k^*)\ar[r] &\displaystyle \bigoplus_{A\in \sB_G}H^2(A,k^*)\ar[r] &\displaystyle \bigoplus_{A\in \sB_G}H^2(A,k_s^*).
}
\end{displaymath}
La fl\`eche fragment\'ee est injective parce que
\[H^1(G,k_s^*/k^*)=\Hom(G,k_s^*/k^*)\inj \bigoplus_{A\in \sB_G}\Hom(A,k_s^*/k^*)=\bigoplus_{A\in \sB_G}H^1(A,k_s^*/k^*). \]

On en d\'eduit l'inclusion $\widetilde{\Br}_{\nr}(BG)\inj \Br_{\nr}(BG_{k_s})$.
\end{proof}

\subsection{Une g\'en\'eralisation de $A^0_{\NR}(-,M_n)$ inspir\'ee par Bogomolov}

\begin{defn}\label{Fneg-st}
Soit $F:\Sm_\fl^\op \to \Ab$. On d\'efinit
\[F_\neg(X):=\bigcup_{U}\Ker(F(X)\to F(U))=\Ker(F(X)\to F(\Spec k(X))) \]
o\`u $U$ parcourt les ouverts de $X$ et $k(X)$ est le corps des fonctions de $X$, et
\[F_\st(X):=F(X)/F_\neg(X)=\IM(F(X)\to F(\Spec k(X))).\]
Ce sont respectivement la \emph{partie n\'egligeable} et l'\emph{image stable} de $F$ (en $X$). Ils d\'efinissent des foncteurs sur $\Sm_\fl$.
\end{defn}

\begin{lemme}
Supposons $k$ infini. Si $F$ est homotopique et pur en coniveau $\geq c$ (\cf \cite[d\'ef. 3.1]{BG1}), alors $F_\neg,F_\st$ le sont aussi.
\end{lemme}

\begin{proof}
Soit $E\to X$ un fibr\'e vectoriel. Soit $U$ un ouvert de $X$. Consid\'erons le diagramme cart\'esien suivant
\begin{displaymath}
\xymatrix{
 E \ar[d]&\supset & E_U \ar[d]&&\ar @{_{(}->}[ll] E_{\eta}\ar[d]& k(E)\ar @{_{(}->}[l]\\
 X  &\supset & U &\supset & \Spec k(X) 
}
\end{displaymath}
D'apr\`es l'hypoth\`ese sur $F$, on a le diagramme commutatif
\begin{displaymath}
\xymatrix{
 F(E) \ar[r] & F(E_U)\ar[r] & F(E_{\eta})\ar[r] &F( k(E))\\
 F(X) \ar[r] \ar[u]_{\wr}& F(U) \ar[r]\ar[u]_{\wr} & F( k(X))\ar[u]_{\wr} 
}
\end{displaymath}
D'o\`u $F_\neg(X)\iso F_\neg(E)'$ o\`u $F_\neg(E)'=\Ker(F(E) \to F(E_{\eta}))$. 
En fait $F_\neg(E)'=F_\neg(E)$. En effet, on va montrer $F(E_{\eta}) \inj F(\Spec k(E))$. On peut se ramener au cas $X=\Spec k$. Pour tout ouvert $V$ de $E_{\eta}=\A^n_k$, $V(k)$ est non vide. Donc $F(k) \to F(V)$ admet une section et $F(k)\inj F(k(E))$. Alors, on a $F_\neg(X) \iso F_\neg(E)$.

Soit $U\subset X$ un ouvert de $X$ de coniveau $\delta(X,U)\geq c$ (\cf \cite[d\'ef. 2.1]{BG1}). On a aussi le diagramme commutatif suivant:
\begin{displaymath}
\xymatrix{
0 \ar[r] & F_\neg(U) \ar[r]& F(U)\ar[r] & F(\Spec k(U))\\
0 \ar[r] & F_\neg(X) \ar[r] \ar @{-->}[u]& F(X) \ar[r] \ar[u]_{\wr}& F(\Spec k(X))\ar[u]_{\wr} 
}
\end{displaymath}
car $k(U)=k(X)$. D'o\`u $F_\neg(X)\iso F_\neg(U)$. Ainsi, $F_\neg$ est homotopique et pur en coniveau $\geq c$.

Donc d'apr\`es la d\'efinition de $F_\st$, on en d\'eduit tout de suite qu'il est homotopique et pur en coniveau $\geq c$.
\end{proof}

\begin{rem}
Si $G$ est un groupe fini sur un corps infini $k$ contenant $\mu_m$ o\`u $m$ est inversible dans $k$, alors $F_\neg(BG),F_\st(BG)$ sont bien d\'efinis (\cf d\'ef. \ref{gp-nr-gm}). 
\end{rem}

\begin{defn}\label{F_NR}
On d\'efinit
\[F_{\NR}(BG):=\{x\in F(BG)\mid \forall (D,g), \partial_{D,g}(x)\in (F_{-1})_\neg(BD)\}.\]
\end{defn}

\begin{ex}[\protect{\cite{bogomolov1}}]
Soit $F=H^i_{\et}(-,\Z(n))$. Pour $X$ lisse, on retrouve des classes $k$-n\'egligeables de $H^i_{\et}(X,\Z(n))$: 
\[H^i_\neg(X,\Z(n))=\Ker(H^i_{\et}(X,\Z(n)) \to H^i_{\et}(k(X),\Z(n))),\]
et la cohomologie stable de $H^i_{\et}(X,\Z(n))$:
\[ H^i_\st(X,\Z(n))=H^i_{\et}(X,\Z(n))/H^i_\neg(X,\Z(n)).\]

Donc d'apr\`es \cite[prop. 3.9]{BG1}, l'exemple \ref{exF_{-1}(X)}, 1) et la d\'efinition \ref{F_NR}, on a
\begin{multline}\label{H^i_NR(BG)}
H^i_{\NR}(BG,\Z(n))\\
=\{x\in H^i_{\et}(BG,\Z(n))\mid \forall (D,g), \partial_{D,g}(x)\in H^{i-1}_\neg(BD,\Z(n-1)).\}
\end{multline}
\end{ex}

\subsection{Th\'eor\`eme de Peyre}

\begin{defn}\label{per-neg}
Soit $G$ un groupe fini sur un corps $k$ contenant $\mu_m$ o\`u $m$ est inversible dans $k$.

On d\'efinit le groupe $H^3_p(G,\Q/\Z)$ des \textit{classes permutation-n\'egligeables} comme le groupe \cite[d\'ef. 4]{peyre1}
\[\sum_{H\subset G}\Cores^G_H(\IM(H^1(H,\Q/\Z)^{\otimes 2} \by{\cup} H^3(H,\Q/\Z))).\]

Notons $H^4_{Ch}(BG,\Z(2))=\left\langle c_2(\rho)\right\rangle$ le sous-groupe de $H^4_{\et}(BG,\Z(2))$ engendr\'e par les classes de Chern des repr\'esentations $\rho$ de $G$.
\end{defn}

Utilisant \cite[prop. 1 et prop. 2]{peyre1} et \cite[th. 7.1]{BG1}, on obtient
\begin{lemme}\label{eg:per-neg}
Si $k$ est s\'eparablement clos, alors
\[H^3_p(G,\Q/\Z)\otimes \Z[\frac{1}{2}] = H^4_{Ch}(BG,\Z(2)) \otimes \Z[\frac{1}{2}].\] 
\end{lemme}

\begin{thm}
Soit $k$ un corps contenant $\mu_m$ o\`u $m$ est inversible dans $k$. Soient $G$ un $k$-groupe fini d'exposant $m$ et $W$ une repr\'esentation fid\`ele de $G$. On a un isomorphisme:
\begin{equation}\label{geneqP}
H^4_{\NR}(BG,\Z(2))/H^4_{Ch}(BG,\Z(2)) \iso H^4_{\nr}(k(W)^G,\Z(2))
\end{equation}
o\`u $H^4_{\NR}(BG,\Z(2))$ est comme en \eqref{H^i_NR(BG)} et $H^4_{Ch}(BG,\Z(2))$ comme dans la d\'efinition \ref{per-neg}. 

\end{thm}

\begin{proof}
D'apr\`es le th\'eor\`eme \ref{thmprincipal}, on a:
\[H^3_{\nr}(k(W)^G,\Q/\Z(2))= A^0_{\NR}(BG,H^3_{\et}(\Q/\Z(2))).\]

La suite exacte de \cite[th. 8.3 b)]{BG1} donne une suite exacte
\[0\to CH^2(BG) \to H^4_{\et}(BG,\Z(2)) \to A^0(BG,H^4_{\et}(\Z(2))) \to 0.\]

Alors, on a le diagramme commutatif suivant
\begin{displaymath}
\xymatrix{
0\ar[r]& CH^2(BG) \ar[r] & H^4_{\et}(BG,\Z(2))\ar[r] \ar[d]^{\partial_{D,g}} &A^0(BG,H^4_{\et}(\Z(2)))\ar[d]^{\partial_{D,g}}\ar[r]& 0\\
&& H^3_{\et}(BD,\Z(1))\ar[r]^{\sim}& A^0(BD,H^3_{\et}(\Z(1)))
}
\end{displaymath}
o\`u l'isomorphisme provient de \cite[th 8.3 a)]{BG1}. D'apr\`es \cite[p. 257]{totaro}, $CH^2(BG)$ est engendr\'e par des classes de Chern des repr\'esentations de $G$. D'o\`u on d\'eduit \eqref{geneqP}.
\end{proof}

\begin{rem}
En particulier, si $k$ est alg\'ebriquement clos de caract\'eristique z\'ero, on a d'apr\`es  \cite[th. 7.1 et lemme 7.3, 1)]{BG1}
\begin{multline*} 
H^4_{\et}(BG,\Z(2))\cong H^3(G,H^1_{\et}(k,\Z(2))=H^3(G,\Q/\Z),\\
\text{et}\ H^3_{\et}(BD,\Z(1))=H^3(D,\Z)=H^2(D,\Q/\Z).
\end{multline*}

Donc d'apr\`es le lemme \ref{eg:per-neg} et \eqref{geneqP}, on a:
\[H^3_{\nr}(G,\Q/\Z)/H^3_p(G,\Q/\Z) \surj H^3_{\nr}(k(W)^G,\Q/\Z(2)),\]
 et son noyau est annul\'e par une puissance de $2$, o\`u 
\[H^3_{\nr}(G,\Q/\Z)=\bigcap_{D\subset G,\ g:I\to Z_G(D)}\Ker(H^3(G,\Q/\Z)\by{\partial_{D,g}} H^2(D,\Q/\Z)).\]

On peut montrer que les r\'esidus $\partial_{D,g}$ sont \'egaux \`a ceux de Peyre  \cite[d\'ef. 5]{peyre1}. On retrouve alors le th\'eor\`eme 1 de \cite{peyre1}.
\end{rem}


\begin{rem}
Gr\^ace \`a la suite exacte \eqref{se-raffine} et au th\'eor\`eme \ref{thm-B}, on obtient une g\'en\'eralisation en degr\'e $3$ de la suite exacte de Bogomolov:
\begin{multline}\label{nab-nr3}
0\to A^0_{\nr}(BG,H^3_{\et}(\Q/\Z(2)))\\ \to A^0_{\nab}(BG,H^3_{\et}(\Q/\Z(2))) \by{\partial_{D,g}} \bigoplus_{D,g} \Br_{\nr}(BD).
\end{multline}
\end{rem}

\end{document}